\documentclass[11pt,reqno]{amsart}
\usepackage{enumerate}
\usepackage{enumitem}
\usepackage{mathrsfs}
\usepackage{url}
\usepackage{cite}
\usepackage{comment}
\usepackage{fancyhdr}
\usepackage{amsfonts, amsmath, amssymb, amscd, amsthm, bm, cancel}
\usepackage{amsrefs}
\usepackage[%backref=page,
linktocpage=true,colorlinks,citecolor=magenta,linkcolor=blue,urlcolor=magenta]{hyperref}
\newtheorem{thm}{Theorem}[section]

\newtheorem{rmk}{Remark}[section]

\newtheorem{lemma}{Lemma}[section]

\numberwithin{equation}{section}
\allowdisplaybreaks

\begin{document}

\title[Curvature measures and volume preserving flows]{Curvature measures and volume preserving flows in h-convex geometry}

\author[X. Zhang]{Xinxiu Zhang}
\address{School of Mathematical Sciences, University of Science and Technology of China, Hefei 230026, P.R. China}
\email{\href{mailto:zhangxinxiu@mail.ustc.edu.cn}{zhangxinxiu@mail.ustc.edu.cn}}

\date{May 14, 2026}
%\thanks{}

\begin{abstract}
In this paper, we introduce a concept of modified curvature measures for h-convex domains in hyperbolic space and prove a shifted version of the generalized Alexandrov theorem, which characterizes geodesic balls in hyperbolic space via modified curvature measures. This result generalizes the classical work of Kohlmann \cite{Kohlmann1998}. We further investigate a volume-preserving shifted $\alpha$-Gauss curvature flow for strictly h-convex hypersurfaces in hyperbolic space and prove that for any $\alpha>0$ the solution exists for all time and converges smoothly to a geodesic sphere. The theory of modified curvature measures developed in the first part serves as a key ingredient in characterizing the Hausdorff limit shape of the flow.
\end{abstract}

\maketitle
\tableofcontents

%-------------------------------------------------------------------------

\section{Introduction}

The classical curvature measures for sets of positive reach in Euclidean space were  introduced by Federer \cite{Federer1959}, which is a concept originating from integral geometry and convex geometry. In \cite{Schneider1979}, Schneider characterized geodesic balls for convex bodies via curvature measures in Euclidean space, which is a result generalizing the classical Alexandrov theorem. Later, Kohlmann \cite{Kohlmann1991} obtained the Steiner formula and extended the notion of the curvature measures to  sets of positive reach in space forms. 
 In addition, in \cite{Kohlmann1998} he obtained a generalized version  of the Alexandrov theorem for convex bodies in space forms. Recently, the curvature measure theory has been applied to characterize the Hausdorff limit shapes of the volume preserving curvature flows, see \cites{ALWX25,AndrewsWei2021, WeiYangZhou2024, WeiYangZhou2025}. 
 
 %For volume preserving flows which are not homogeneous of degree one, the classical argument for smooth convergence of the solutions does not work due to the lack of uniform lower bound on curvature. One technique to overcome this difficulty is to establish the subsequential Hausdorff convergence via the monotonicity of the flows and the subsequential limit shape can be characterized using the generalized Alexandrov theorem in the curvature measure theory.  

In this paper we consider curvature measures and volume preserving curvature flows for h-convex domains in hyperbolic space $\mathbb{H}^{n+1}$. Recall that a compact domain in hyperbolic space is called h-convex if it is convex by horospheres. It is known that a smooth compact domain is h-convex  (resp. strictly h-convex) if and only if all principal curvatures of its boundary are bounded below by $1$ (resp. strictly greater than $1$). The h-convex geometry in hyperbolic space has attracted considerable interest. We refer the readers to \cite{Espinar09,AndrewsChenWei2021,Li-Xu} for the fundamental results on h-convex geometry. 

  We establish two main results:
      
      (1) We first introduce a concept of modified curvature measures for h-convex domains in hyperbolic space, and then prove a generalized Alexandrov type theorem  to characterize geodesic balls via the modified hyperbolic curvature measures. This generalizes the classical result of Kohlmann \cite{Kohlmann1998}.  
      
      (2) We study volume preserving shifted $\alpha$-Gauss curvature flows of h-convex domains for any power $\alpha>0$, and show that the solution exists for all time and converges smoothly to a geodesic sphere. The modified curvature measure theory of h-convex domains established in the first part will be a key ingredient to characterize the Hausdorff limit of the flow. 

     In the following we present our results in more detail.

\subsection{Modified curvature measures in hyperbolic space} We denote by $\mathcal{K}_{n+1}(\mathbb{H}^{n+1})$ the collection of  compact, convex subsets of  $\mathbb{H}^{n+1}$ with nonempty interior.  For any $A\in\mathcal{K}_{n+1}(\mathbb{H}^{n+1})$, Kohlmann \cite{Kohlmann1991} proved that the generalized principal curvatures $\kappa_i(v)\in[0,\infty]\,(1\leq i\leq n)$ are  well defined for $\mathcal{L}^n$-almost every $v$ in the unit normal bundle $\mathcal{N}_A$ of $A$, which is an $n$-dimensional strong Lipschitz submanifold of $T\mathbb{H}^{n+1}$. By the convexity of $A$,  for any $x\in \mathbb{H}^{n+1}$ there exists unique $x'\in A$ such that $d_{\mathbb{H}^{n+1}}\left(x', x\right)=d_{\mathbb{H}^{n+1}}(x, A)$. Hence we can define the projection map $x\mapsto x'=:f_{A}$.  For any $\varepsilon>0$,   the parallel set of $A$ is defined by
$$A_{\varepsilon}=\left\{x \in \mathbb{H}^{n+1} : d_{\mathbb{H}^{n+1}}(x, A) \leq \varepsilon\right\}.
$$ 
Kohlmann \cite{Kohlmann1991} proved  there exist Borel measures $V_0(A, \cdot), \cdots, V_{n+1}(A, \cdot)$ on $\mathbb{H}^{n+1}$ such that for any Borel set $ \beta \in \mathcal{B}\left(\mathbb{H}^{n+1}\right)$ there holds the following Steiner-type formula:
\begin{align}
\text{Vol}_{\mathbb{H}^{n+1}}(f_A^{-1}(\beta)\cap A_\varepsilon)=\sum_{k=0}^{n+1} \ell_{n+1-k}(\varepsilon) V_k(A, \beta),\label{steiner}
\end{align}
where 
$$\ell_{k+1}(t)=\int_0^t \sinh ^k(s) \cosh ^{n-k}(s) d s, \quad k=0,1, \cdots, n$$ 
and $\ell_0\equiv1$.  The coefficients $V_0(A, \cdot), \cdots, V_{n+1}(A, \cdot)$  are called the curvature measures of the convex body $A$ that can be represented as \begin{align*}&V_k(A, \beta):=\binom{n}{k}\int_{\Pi^{-1}(\beta) \cap \mathcal{N}_A} g(v) {E}_{n-k}(v) d \mathcal{L}^n v \quad \text { for }\,0\leq k \leq n, \\ &V_{n+1}(A, \beta):=\mathcal{L}^{n+1}(A \cap \beta),\end{align*}
Here $E_k(v):=\binom{n}{k}^{-1} \sum\limits_{1\leq i_1<\ldots<i_k \leq n} \kappa_{i_1}(v)\cdot\ldots\cdot \kappa_{i_k}(v)$ for $1\leq k\leq n$ and $E_0(v)=1$, $\Pi$ is the  canonical footpoint map from $T\mathbb{H}^{n+1}$ to $\mathbb{H}^{n+1}$  and $$g(v):= \prod_{i=1}^n\left(1+\kappa_i(v)^2\right)^{-1/2}.$$ In section \ref{preli} we will describe the generalized principal curvatures and $g(v)$ in more detail.  When the boundary $\partial A$ is smooth, $V_k(A, \cdot)$ can be expressed as
$$V_k(A, \beta)=\binom{n}{k}\int_{\partial A \cap \beta} E_{n-k}(\kappa) d \mu,$$
where $E_{n-k}(\kappa)$ is the $(n-k)$-th mean curvature of $\partial A$. When $\beta$ is the whole space, we call $V_k(A)=V_k(A,\mathbb{H}^{n+1})$ the curvature integrals of $A$. 

The following characterization of geodesic balls via the hyperbolic curvature measures was proved by Kohlmann \cite{Kohlmann1998} and can be viewed as a generalization of the classical Alexandrov theorem for convex domains in differential geometry.
\begin{thm}[\cite{Kohlmann1998}]\label{kohlmann98thm}
Let $n \geq 2$ and $k \in\{2, \ldots, n\}$. Assume that $A\in\mathcal{K}_{n+1}(\mathbb{H}^{n+1})$  satisfies 
$$V_{n-k}(A, \beta)=c V_n(A, \beta)$$
for any Borel set $\beta \in \mathcal{B}\left(\mathbb{H}^{n+1}\right)$, where $c>0$ is a constant. Then $A$ is a geodesic ball.
\end{thm}
Theorem \ref{kohlmann98thm} has been applied recently by  Wei, Yang and Zhou \cites{WeiYangZhou2024, WeiYangZhou2025} to characterize  the Hausdorff limit shape of the volume preserving Gauss curvature flows by exploiting  the monotonicity of the $n$-th quermassintegral of the enclosed domain along the flows. 

In this paper,  we consider the geometry of h-convex domains. Denote by $\mathcal{H}_{n+1}(\mathbb{H}^{n+1})$ the collection of nonempty, compact, h-convex domains of $\mathbb{H}^{n+1}$. For any $A\in\mathcal{H}_{n+1}(\mathbb{H}^{n+1})$ and $\beta\in\mathcal{B}(\mathbb{H}^{n+1})$, we introduce the modified curvature measures of $A$ by 
\begin{align}&\tilde V_r(A, \beta)=\binom{n}{r} \int_{\Pi^{-1}(\beta) \cap \mathcal{N}_A} g(v) \tilde{E}_{n-r}(v) d \mathcal{L}^n v\nonumber\, \text { for }\,0\leq r \leq n,\nonumber\\
%&=\binom{n}{r}\sum\limits_{i=0}\limits^{r}(-1)^{r-i}\binom{r}{i}\binom{n}{i}^{-1}V_r(A,\beta),
&\tilde{V}_{n+1}(A, \beta)=\mathcal{L}^{n+1}(A \cap \beta),
\label{tildeV}
\end{align}
where $\tilde{E}_r(v)=\binom{n}{r}^{-1} \sum\limits_{1\leq i_1<\ldots<i_r \leq n} (\kappa_{i_1}(v)-1) \cdot \ldots \cdot (\kappa_{i_r}(v)-1)$ for $1\leq r\leq n$ and $\tilde{E}_0(v)=1$. 
We establish the following shifted version of Theorem \ref{kohlmann98thm}.
\begin{thm} \label{modifiedcurvature}
Let $n \geq2$ and $k \in\{2, \ldots, n\}$. Assume that $A\in\mathcal{H}_{n+1}(\mathbb{H}^{n+1})$ satisfying $$\tilde{V}_{n-k}(A, \beta)=c\tilde{V}_n(A, \beta)$$ for any Borel set $\beta \in \mathcal{B}\left(\mathbb{H}^{n+1}\right)$, where $c>0$ is a constant. Then $A$ is a geodesic ball.
\end{thm}

%\begin{rmk}
%    In \cite{HuWeiZhou2024} Hu, Wei and Zhou established a shifted version of Heintze-Karcher type inequality in hyperbolic space and as a corollary they obtained a shifted version of the classical Alexandrov theorem for smooth h-convex hypersurfaces. Noticing that $\tilde{V}_{n-1}(A,\cdot)=V_{n-1}(A,\cdot)-V_n(A,\cdot)$, we have the theorem \ref{modifiedcurvature} holds for the case $k=n-1$ by the corollary 3.4 in \cite{kohlmannstability}. Together with our results, this yields a generalized version of theorem 1.2 in \cite{HuWeiZhou2024}. 
%\end{rmk}

The framework of the proof is inspired by the  methods developed by Kohlmann in \cite{Kohlmann1998}. To accommodate the modified curvature measures we employ a modified version of the Minkowski formula together with considering the characterizations of generalized principal curvatures of h-convex domains.

\subsection{Volume preserving shifted Gauss curvature flows}

The volume preserving mean curvature flow 
\begin{align}
    \frac{\partial}{\partial t} X(x, t)=(\phi(t)-H) \nu(x, t)\label{huisken87}
\end{align}
was introduced by Huisken \cite{Huisken1987} with strictly convex initial hypersurfaces in Euclidean space. It is proved that the solution to the flow (\ref{huisken87}) exists for all time and  converges smoothly to a round sphere. Subsequently, there have been studies about mixed volume preserving curvature flows in Euclidean spaces  replacing the mean curvature by $1$-homogeneous curvature functions or $\alpha>0$ power of the $k$th mean curvature $E_k(1\leq k\leq n)$,  see \cites{Mccoy2004,Mccoy2005,Mccoy2017,CabezasSinestrari2010,Sinestrari2015,AndrewsWei2021}  for instance.  
 
%For nonhomogeneous version of flow (\ref{huisken87}), Bertini and Sinestrari \cite{Bertini2018} considered the volume preserving nonhomogeneous mean curvature flows with strictly convex initial values.

The extension of the flow \eqref{huisken87} in hyperbolic space is challenging due to the negative curvature of the ambient space.  In 2007, Cabezas-Rivas and Miquel \cite{CabezasMiquel2007} extended the volume preserving mean curvature flow \eqref{huisken87} to hyperbolic space by imposing the h-convexity on the initial hypersurface to overcome the difficulties caused by the negative curvature. Since then, many subsequent works retained the h-convexity hypothesis for initial data. For instance, Wang and Xia \cite{WangXia2014} considered the 1-homogeneous quermassintegral preserving flow with h-convex initial hypersurfaces and obtained the Alexandrov-Fenchel inequalities for smooth h-convex domains. Besides 1-homogeneous flows, there are also studies on  $\alpha$-homogeneous volume preserving flows (see \cites{Makowski2012, GuoLiWu2017, AndrewsWei2018}) and nonhomogeneous flows (see \cite{Bertini2017}), among which most require that the initial hypersurfaces are at least h-convex. 

More recently, there have been works that require weaker convexity conditions for initial hypersurfaces in hyperbolic space. Andrews, Chen and Wei \cite{AndrewsChenWei2021} considered the volume preserving flows by powers of $k$th mean curvature, requiring positive intrinsic sectional curvatures on the initial hypersurfaces; Wei, Yang and Zhou \cites{WeiYangZhou2024, WeiYangZhou2025} considered the volume preserving flows by powers of Gauss curvature with strictly convex initial hypersurfaces. 

In addition to relaxing the convexity conditions of initial hypersurfaces for flows, one can also consider speed functions with different scaling rates. Andrews, Chen and Wei \cite{AndrewsChenWei2021} introduced the modified quermassintegrals of smooth h-convex domains and  considered the modified quermassintegral preserving flows by curvature functions which are $1$-homogeneous of the shifted principal curvatures
\begin{align}
\frac{\partial}{\partial t} X(x, t) =\left(\phi(t)-F(\mathcal{W}-I) \right)\nu(x, t).\label{jems}
\end{align}
As an application, they established Alexandrov-Fenchel type inequalities of modified quermassintegrals for smooth h-convex domains. There are many other works on curvature flows in hyperbolic space involving shifted curvatures, see \cites{And94,BH17,HuLiWei2022,GaoZhou2024,Lyn22,WangWeiZhou2023,ji2025} for instance. 

In the second part of this paper, we consider the volume preserving shifted $\alpha$-Gauss curvature flow of strictly h-convex hypersurfaces in hyperbolic space. Let $X_0: M^n\rightarrow\mathbb{H}^{n+1}$ be a smooth embedding such that $M_0:=X_0(M^n)$ is a closed hypersurface with nonempty compact domain $\Omega_0$ as its interior in hyperbolic space $\mathbb{H}^{n+1}$.  We consider a smooth family of immersions $X: M^n \times[0, T) \rightarrow \mathbb{H}^{n+1}$ satisfying 
\begin{align}
\left\{\begin{aligned} 
\frac{\partial}{\partial t} X(x, t) & =\left(\phi(t)-\tilde{E}^{\frac{\alpha}{n}}_n(x,t)\right) \nu(x, t), \\ 
X(\cdot,0) & =X_0(\cdot),\end{aligned}\right.\label{flow}
\end{align}
where $\alpha>0$, $\nu(x, t)$ is the unit outer normal of the hypersurface $M_t:=X(M^n,t)=\partial \Omega_t$, $\mathcal{W}$ is the Weingarten matrix of $M_t$, $\tilde{E}_n(x,t):=E_n(\mathcal{W}-I)$ is the shifted Gauss curvature of $M_t$ and  \begin{align}
    \phi(t):=\frac{\int_{M_t}\tilde{E}^{\frac{\alpha}{n}}_n(\mathcal{W}) \mathrm{d} \mu_{g_t} }{|M_t|} \label{phi}\end{align} is the global term chosen to preserve the volume of the enclosed domain.  We prove the following convergence result of the flow \eqref{flow}. 

    \begin{thm}\label{mainthm}
   Let  $M_0$ be a smooth closed hypersurface in $\mathbb{H}^{n+1}$ enclosing a strictly h-convex compact domain $\Omega_0$. Then, for any $\alpha>0$,  the flow \eqref{flow}  has a strictly h-convex solution $M_t$ for all time $t\in [0,\infty)$, and $M_t$ converges smoothly and exponentially to a geodesic sphere that encloses the same volume as $\Omega_0$ as $t\rightarrow +\infty$.
\end{thm}

%To the best of our knowledge, this is the first result on global constrained shifted curvature flows which are not homogeneous of degree one. 
For the long time existence of \eqref{flow},  we prove the lower bound estimate of shifted curvatures by analyzing the auxiliary function  $$P(p,t)=\log{b}+ A\cosh d_{p_0}-Bu,$$
where $b$ is the largest reciprocal of the shifted principal curvatures $\kappa_i-1, i=1,\dots,n$, $d_{p_0}$ and $u$ are respectively the radial function and support function of $M_t$ when $M_t$ is star shaped with respect to some point $p_0$. 
Then we adapt Tso's technique \cite{Tso} to obtain an upper bound on shifted Gauss curvature, which implies the long time existence by classical arguments.

 The convergence of $1$-homogeneous shifted curvature flows \eqref{jems} in \cite{AndrewsChenWei2021} is based on establishing curvature pinching estimate along the flow. When the degree of the homogeneity is any $\alpha\neq 1$, the method to obtain the curvature pinching estimates does not work. In our case, to overcome this obstacle we adapt the technique of Wei, Yang and Zhou \cites{WeiYangZhou2024, WeiYangZhou2025} to characterize the Hausdorff limit shape of the flow \eqref{flow}. The theory of modified curvature measures established in Theorem \ref{modifiedcurvature} will be a key ingredient. 
 
 %Hence we need to establish a modified version of the theorem \ref{kohlmann98thm}.  Moreover, in order to obtain the subsequential convergence of curvature measures, we need to obtain a uniform lower bound for the integral average of shifted Gauss curvature of solutions. Thus we derive the relations between modified quermassintegrals and modified curvature integrals.

\subsection{Organization of this paper} 
This paper is organized as follows: 
In section \ref{preli} we recall some preliminaries including the quermassintegrals in hyperbolic space and the modified quermassintegrals of h-convex domains introduced by Andrews, Chen and Wei \cite{AndrewsChenWei2021}. Moreover we recall the preliminaries about the geometry of h-convex domains in hyperbolic space, the generalized principal curvatures introduced by Kohlmann \cite{Kohlmann1991} and the evolution equations of curvature flows. In section \ref{rigidity}, we consider the convergence behavior of compact h-convex domains in the Hausdorff metric and obtain the characterization of geodesic balls via modified curvature measures for compact h-convex domains.% by adapting technique of Kohlmann \cite{Kohlmann1998} 
Moreover, we recall some relations between modified quermassintegrals and modified curvature measures. 
In section \ref{SGCF}, we give the long time existence of solutions and convergence of the flow (\ref{flow}).
In subsection \ref{prio}, we derive a priori estimates of the flow depending on $T$ on finite time interval $[0,T)$ and obtain the long time existence of the flow \eqref{flow}.
In subsection \ref{convergenceofflow}, by the expression of $n$-th shifted curvature measures of smooth h-convex domains and the uniform estimates of inner radius and outer radius we can obtain the uniform lower bound on the average of shifted Gauss curvature. Then we obtain the smooth and exponential convergence of flow \eqref{flow} by the convergence and rigidity  result we derived in section \ref{rigidity}. 

\section{Preliminaries}\label{preli}
In this section, we recall some preliminary results concerning evolution equations of curvature flows, hyperbolic geometry and curvature measure theory. 

We first introduce some notation.  For any smooth closed hypersurface $M$ which is the boundary of a bounded domain of $\mathbb{H}^{n+1}$, we denote by $g=\{g_{ij}\}$, $\mathcal{W}=\{h^j_i\}$ ,$\nabla$, $\overline{\nabla}$ the induced metric and Weingarten matrix, induced Levi-Civita connection of $M$ and  Levi-Civita connection of $\mathbb{H}^{n+1}$ respectively. We denote by $B(x,r)$ the geodesic ball centered at $x$ of radius $r$ in $\mathbb{H}^{n+1}$.
In the arguments below, we symmetrize the indices of every tensor appearing for convenience. For example,  $h^j_i$ is taken to mean its symmetrized form $\frac{1}{2}(h^j_i + h^i_j)$. Along the flow (\ref{flow}) we denote by $\Omega_t$ the interior enclosed by $M_t$. For $p\in\mathbb{H}^{n+1}$, let $d_p(x)=d_{\mathbb{H}^{n+1}}(p,x)$. 

We denote by $\mathcal{K}(\mathbb{H}^{n+1})(\mathcal{K}(\mathbb{R}^{n+1}))$  the collection of nonempty, compact, convex subsets in $\mathbb{H}^{n+1}(\mathbb{R}^{n+1})$. And denote by $\mathcal{H}_{n+1}(\mathbb{H}^{n+1})$ the collection of nonempty, compact, h-convex domains in $\mathbb{H}^{n+1}$.

In the following we equip the tangent bundle $T\mathbb{H}^{n+1}$ with the Sasaki metric with respect to the metric on $\mathbb{H}^{n+1}$ denoted by $\langle \cdot, \cdot\rangle_{\mathbb{H}^{n+1}}$. Denote by $\Pi :T\mathbb{H}^{n+1}\rightarrow \mathbb{H}^{n+1}$ the canonical projection and by $\Gamma:TT\mathbb{H}^{n+1}\rightarrow T\mathbb{H}^{n+1}$ the connection map. For any $w\in T\mathbb{H}^{n+1}$ we identify $T_wT\mathbb{H}^{n+1}$ with $T_{\Pi(w)}\mathbb{H}^{n+1}\times T_{\Pi(w)}\mathbb{H}^{n+1}$ by the isometry $\Pi_*\times\Gamma$.

\subsection{Quermassintegrals in hyperbolic space}
For any $A$ in $\mathcal{K}_{n+1}(\mathbb{H}^{n+1})$, the quermassintegrals of $A$ are defined as (see \cites{integralgeometrybook, integralgauss}) 
$$W_k(A)=\frac{\omega_{k-1} \cdots \omega_0}{\omega_{n-1} \cdots \omega_{n-k}} \int_{\mathcal{L}_k} \chi\left(L_k \cap A\right) d L_k, \quad k=1, \ldots, n,$$
where $\mathcal{L}_k$ is the space of $k$-dimensional totally geodesic subspaces in $\mathbb{H}^{n+1}$, $\omega_n$ denotes the area of the $n$-dimensional unit sphere in Euclidean space, and $\chi$ is the Euler characteristic. By the convexity of $A$, we know that $\chi$ is defined to be $1$ if $\mathcal{L}_k \cap A \neq \emptyset$ and to be $0$ otherwise. Furthermore, set $$W_0(A)=|A|, \quad W_{n+1}(A)=\left|\mathbb{B}^{n+1}\right|=\frac{\omega_n}{n+1} .$$ When $A$ is smooth, the quermassintegrals and the curvature integrals of $A$ are related by the following equations (see \cite{integralgauss}): \begin{align}
\binom{n}{n-k}^{-1}V_{n-k}(A) & =(n-k) W_{k+1}(A)+k W_{k-1}(A), \quad k=1, \ldots, n-1, \nonumber\\
V_n(A) & =n W_1(A)=|\partial A|\nonumber \\
V_0(A) & =\omega_n+n W_{n-1}(A) .\label{QandC}
\end{align}
Wang and Xia \cite{WangXia2014} obtained the Alexandrov-Fenchel inequality for bounded smooth h-convex domains  in hyperbolic space which can be expressed as 
\begin{align}
 W_k(A) \geq f_k \circ f_{\ell}^{-1}\left(W_{\ell}(A)\right)   
\end{align}
for any $0\leq \ell<k\leq n$. Here $f_k:[0, \infty) \rightarrow \mathbb{R}_{+}$ is a monotone
function defined by $f_k(r)=W_k\left(B_r\right)$, and equality holds if and only if $A$ is a geodesic ball.

The quermassintegrals are monotone with respect to inclusion of convex bodies. 
\begin{thm}\label{thmmonoofquer}
    Suppose that $A, B\in\mathcal{K}_{n+1}(\mathbb{H}^{n+1})$. If $A\subseteq B$,     then ${W}_k\left(A\right) \leq {W}_k\left(B\right)$. 
\end{thm}
Andrews, Chen and Wei \cite{AndrewsChenWei2021} introduced modified quermassintegrals of bounded smooth h-convex domains. More precisely, suppose that $A\in \mathcal{H}_{n+1}(\mathbb{H}^{n+1})$ is a smooth h-convex domain and denote by 
\begin{align}\tilde{W}_k(A):=\sum_{i=0}^k(-1)^{k-i}\binom{k}{i} W_i(A)\label{Quermassine}\end{align} the modified quermassintegrals of $A$. They established the monotonicity of the modified quermassintegrals as follows. 
\begin{thm}[\cite{AndrewsChenWei2021}]\label{mono}
    If $A$ and $B$ are h-convex domains with $A\subset B$, then $\tilde{W}_k\left(A\right) \leq \tilde{W}_k\left(B\right)$.
\end{thm}
Moreover, they obtained the Alexandrov-Fenchel type inequality for modified quermassintegrals of smooth h-convex domains.
\begin{thm}[\cite{AndrewsChenWei2021}]\label{modifiedalexandrov}
For any  smooth, closed and strictly h-convex hypersurface $M=\partial \Omega$ in $\mathbb{H}^{n+1}$ and $0\leq l<k \leq n$ we have
$$\tilde{W}_k(\Omega) \geq \tilde{f}_k \circ \tilde{f}_l^{-1}\left(\tilde{W}_l(\Omega)\right).
$$
Equality holds if and only if $\Omega$ is a geodesic ball. Here the function $\tilde{f}_k: [0,\infty) \rightarrow \mathbb{R}_{+}$ is defined by $\tilde{f}_k(r)=\tilde{W}_k(B(r))$.
\end{thm}

\subsection{Hyperbolic geometry}
%As is mentioned in \cite{AndrewsChenWei2021}, a domain in hyperbolic space $\mathbb{H}^{n+1}$ is called h-convex if it can be expressed as the intersection of a family of horoballs.  
For any $A\in\mathcal{H}_{n+1}(\mathbb{H}^{n+1})$, the inner radius $\rho_{-}(A)$ and the outer radius $\rho_{+}(A)$ of $A$ are defined by $$\rho_{-}(A)=\sup \left\{\rho: B_\rho(p) \subset A \text { for some } p \in \mathbb{H}^{n+1}\right\}$$ and $$\rho_{+}(A)=\inf \left\{\rho: A \subset B_\rho(p) \text { for some } p \in \mathbb{H}^{n+1}\right\}.$$

It is known that for any $A\in\mathcal{H}_{n+1}(\mathbb{H}^{n+1})$ the inner radius $\rho_{-}(A)$ and the outer radius $\rho_{+}(A)$ can be compared as follows:

\begin{thm}(\cites{BorisenkoGallegoReventos2001, BorisenkoMiquel1999,BorisenkoMiquel2002,Makowski2012,CabezasMiquel2007})\label{hconvex}
    For any $A\in\mathcal{H}_{n+1}(\mathbb{H}^{n+1})$, let $o$ be the center of an inner ball of $A$. We have

    (1) The maximum of distance from $o$ to $\partial A$ is controlled by $\rho_-(A)$
    \begin{align}
        \max_{p \in \partial A} d_{\mathbb{H}^{n+1}}(o, p) \leq \rho_{-}(A)+\ln \frac{\left(1+\sqrt{\tanh \rho_{-}(A) /2}\right)^2}{1+\tanh \rho_{-}(A) /2}<\rho_{-}(A)+\ln2.\label{innerandouter}
    \end{align}
    Moreover, there exists a constant $c>0$ such that $$\rho_{+}(A) \leq c\left(\rho_{-}(A)+\rho_{-}^{1/2}(A)\right)$$.

    (2)  For any $p\in \mathring{A}$, we have \begin{align}
        \langle \partial d_{p},\nu\rangle_{\mathbb{H}^{n+1}} \geq \tanh \left(\mathrm{d}_{\mathbb{H}}(p, \partial A)\right), \text{for any $q\in\partial A$}.\label{support}
    \end{align}
\end{thm}

Recall that in a metric space $(\mathcal{X},d_{\mathcal{X}})$,  the $\varepsilon$-parallel set of $A\subseteq \mathcal{X}$ is defined  by $$A^\varepsilon:=\{x\in\mathcal{X}: d_\mathcal{X}(x,A)\leq \varepsilon\}.$$
Denote by $\mathcal{C}(\mathcal{X})$ the collection of nonempty compact subsets in $\mathcal{X}$. The Hausdorff metric on $\mathcal{C}(\mathcal{X})$ between two compact subsets $A$ and $B$ is defined by $$d_\mathcal{X}(A,B):=\inf\{\varepsilon\geq 0: A\subseteq B^\varepsilon\, \text{and} \,B\subseteq A^\varepsilon\}.$$

In Euclidean space, $(\mathcal{C}(\mathbb{R}^{n+1}),d_{\mathbb{R}^{n+1}})$ is complete and any closed, bounded subset of $\mathcal{C}(\mathbb{R}^{n+1})$ is compact, from which we know that every sequence in $\mathcal{K}(\mathbb{R}^{n+1})$ has a convergent subsequence (see \cite{schneiderbook}). 

In $\mathbb{H}^{n+1}$, it is  known that $(\mathcal{C}(\mathbb{H}^{n+1}),d_{\mathbb{H}^{n+1}})$ is complete and every closed, bounded subset of $\mathcal{C}(\mathbb{H}^{n+1})$ is compact. Now we consider the Klein model of $\mathbb{H}^{n+1}$. In this model the hyperbolic geodesic between two points coincide with the Euclidean straight line segment.  Consequently, we know that the hyperbolic convexity of subsets of the unit ball is equivalent to Euclidean convexity in the Klein model.  This implies that any bounded sequence in $\mathcal{K}(\mathbb{H}^{n+1})$ has a convergent subsequence.

\subsection{Generalized principal curvatures}\label{plementaryofgpc}
In this subsection we recall some basic concepts introduced by Kohlmann \cites{Kohlmann1991,Kohlmann1994,Kohlmann1998} and introduce the notion of modified generalized principal curvatures. 
For any $A\in\mathcal{K}(\mathbb{H}^{n+1})$, we define $F_{A}: \mathbb{H}^{n+1} \setminus A \rightarrow T \mathbb{H}^{n+1}$ as
\[ x=\exp_{f_{A}(x)}^{\mathbb{H}^{n+1}}\left(d_{\mathbb{H}^{n+1}}(A, x) F_{A}(x)\right),
\]
where $d_{\mathbb{H}^{n+1}}(x,y)$ denotes  the distance between $x,y\in\mathbb{H}^{n+1}$.
Denote by $\mathcal{N}(A)$ the image of $F_A$, which is a strong Lipschitz submanifold of the tangent bundle $T\mathbb{H}^{n+1}$ \cite{Walter}.

  The regularity properties of $f_A$, $F_A$ and $\mathcal{N}_A$ we introduce below were introduced in \cites{Walter1974 ,Kohlmann1991}. Firstly,   $f_A$  is locally Lipschitz in $\mathbb{H}^{n+1}\setminus A$, which implies that $f_A$ and $F_A$ are differentiable almost everywhere in $\mathbb{H}^{n+1}\setminus A$. For any $p,q\in\mathbb{H}^{n+1}$, we denote by $\overline{p q}$ the set of points of geodesic between two points $p$ and $q $, and denote by $\mathbb{P}_{\overline{p q}}$ the parallel displacement from $p$ to $q$ along $\overline{p q}$. If $f_A$ or $F_A$ are differentiable at $p\in \mathbb{H}^{n+1}\setminus A$, then  $f_A$ and $F_A$ are differentiable at every point of the geodesic ray starting from $f_A(p)$ towards $p$.  

For any $\varepsilon>0$, $\partial A^\varepsilon$ is a $C^1$ submanifold of $\mathbb{H}^{n+1}$.  $\mathcal{N}_A$ is a  $n$-dimensional strong Lipschitz submanifold of $T\mathbb{H}^{n+1}$ and $F_A\mid_{\partial A^\varepsilon}: \partial A^\varepsilon\rightarrow\mathcal{N}_A$ is a local bi-Lipschitz homeomorphism. We denote the dilatation with center $A$ by $h_t: \mathbb{H}^{n+1}\setminus A \rightarrow \mathbb{H}^{n+1}\setminus A$ for $t \in[0,1]$, which is defined by $h_t(q) \in \overline{q f_A(q)}$ and $d_{\mathbb{H}^{n+1}}\left(f_A(q), h_t(q)\right)=t d_{\mathbb{H}^{n+1}}\left(f_A(q), q\right)$. In fact,  $h_t(q)=\Phi_{td_{\mathbb{H}^{n+1}}(A,q)}(F_A(q))$ for any $q\in\mathbb{H}^{n+1}\setminus A$, here $\Phi$ is the geodesic flow

We now recall the generalized principal curvatures of $A$ introduced by Kohlmann \cite{Kohlmann1991}. Denote by $U\subseteq \mathbb{H}^{n+1}\setminus A$ the set consisting of all points in $\mathbb{H}^{n+1}\setminus A$ at which both $f_A$ and $F_A$ are differentiable. Since $f_A$ and $F_A$ are locally Lipschitz, we know that $U$ is a subset of full measure of $\mathbb{H}^{n+1}\setminus A$.  For $p\in\mathbb{H}^{n+1}\setminus A$, we define $N(p):=\Phi_{d(A,p)} (F_A(p))$. By  definition, $N$ is differentiable in $U$. For any $\varepsilon>0$, suppose that $p\in \partial A^\varepsilon\cap U$. Then $N(p)$ is the outer unit normal vector of $\partial A^\varepsilon$ and the classical second fundamental form $\langle D N(\cdot), \cdot\rangle_{\mathbb{H}^{n+1}}$ of $\partial A^\varepsilon$ is well defined at $p$. 
We choose an orthonormal frame $\{u_i\}$ consisting  of eigenvectors of $\langle D N(\cdot),\cdot\rangle_{\mathbb{H}^{n+1}}$ at $p$. By \cite{Kohlmann1991}, $\{h_{\frac{\varepsilon_1}{\varepsilon}*p}u_i\}(\varepsilon_1<\varepsilon)$ is an orthonormal frame consisting of eigenvectors of $\langle D N(\cdot),\cdot\rangle_{\mathbb{H}^{n+1}}$ 
 at $h_\frac{\varepsilon_1}{\varepsilon}(p)$. Let $\kappa_i(p):=\langle D N(u_i),u_i\rangle_{\mathbb{H}^{n+1}}$ and $f_i(t)=\sinh{(t)}\kappa_i(p)+\cosh{(t)}$. Following \cite{Kohlmann1991}  the generalized principal curvatures are defined by 
 \begin{align}\kappa_i(F_A(p))=\lim\limits_{t\rightarrow 0^+}\frac{f'_i(t-\varepsilon)}{f_i(t-\varepsilon)}\in \mathbb{R}\cup\{\infty\}\label{geprincur}\end{align} 
 for any $p\in \partial A^\varepsilon\cap U$, where $\kappa_i(F_A(p))=\infty$ iff $f_i(-\varepsilon)=0$. When $A\in\mathcal{K}_{n+1}({\mathbb{H}^{n+1}})$ and $A$ is smooth, at any $p\in\partial A$ we know that there exists unique $v\in \mathcal{N}_A\cap T_p\mathbb{H}^{n+1}$ and $\kappa_i(v)$ coincides with the usual principal curvatures of $\partial A$ at $p$. Denote by $\tilde{\mathcal{N}}_{A}\subseteq \mathcal{N}_A$ the set on which $\kappa_i(1\leq i\leq n)$ are well defined. By $\mathcal{N}_A=\Phi_{-\varepsilon}(N(\partial A^\varepsilon))$ we know that $\tilde{\mathcal{N}}_{A}$ is a $\mathcal{L}^n$ full measure subset of $\mathcal{N}_{A}$. Due to the structure of $A$, by \cite{Kohlmann1991} we know that $\Pi(\tilde{\mathcal{N}}_{A,0})\subseteq \partial_0 A$ and $\Pi|_{\tilde{\mathcal{N}}_{A,0}}$ is an injection. 
 %\begin{rmk}
  %   In $\mathbb{R}^{n+1}$, for any $A\in\mathcal{K}_{n+1}(\mathbb{R}^{n+1})$, we can define the generalized principal curvatures $k^0_i(p)$ by a similar argument as above, with replacing $\cosh t$ by $1$ and $\sinh t$ by $t$, see  \cite{Kohlmann1991} for details.
% \end{rmk}

Now we recall more details of the generalized curvature measures. In the following we assume that $\kappa_1(v)\geq \kappa_2(v)\geq\dots\geq \kappa_n(v)$ for any $v\in \tilde{\mathcal{N}}_A$ such that the $\kappa_i(v)$ exists. For $0\leq r \leq n$, let 
\[\tilde{\mathcal{N}}_{A, r}:=\left\{v \in \tilde{\mathcal{N}}_A \mid \kappa_1(v)=\ldots=\kappa_r(v)=\infty \text { and } \kappa_{r+1}(v)<\infty\right\} .\]
 When $ v \in \tilde{\mathcal{N}}_{A, r}$, $g(v)$ has following expression: $$g(v):=\infty^{-r} \prod_{i=r+1}^n\left(1+\kappa_i(v)^2\right)^{-1/2},$$
where the ``$\infty^{-r}$" is an algebraic expression. For example, for $v\in \tilde{\mathcal{N}}_{A,r}$, we have \[g(v)E_k(v)=\left\{\begin{aligned}
&0, & k<r,\\
&\binom{n}{r}^{-1}\prod_{i=r+1}^n \left(1+\kappa_i(v)^2\right)^{-1/2}, &k=r,\\
&\binom{n}{k}^{-1}\prod_{i=r+1}^n \left(1+\kappa_i(v)^2\right)^{-1/2}\\
&\qquad\cdot \sum\limits_{\substack{r+1\leq i_{r+1}<\cdots<i_k\leq n}}\kappa_{i_{r+1}}(v)\cdots\kappa_{i_k}(v), &r<k\leq n.\end{aligned}\right.\]
For any $v\in\tilde{\mathcal{N}}_A$, $g(v)E_r(v)$ is a real number for any $r=0,\dots, n$.  If $v\in\tilde{\mathcal{N}}_A\setminus\tilde{\mathcal{N}}_{A,0}$, then  $g(v)=0$. By Kohlmann \cite{Kohlmann1991}, for each $k=0,\dots,n$, the function $v\rightarrow g(v)E_k(v)$ is $\mathcal{L}^n$ measurable on $\tilde{\mathcal{N}}_A$, and $V_k(A,\cdot)$ is a Borel measure.

Now we introduce the modified curvature measures. In the case where $\kappa_i=\infty$ we treat $\kappa_i-1$ as $\infty$ algebraically. For $v\in\tilde{\mathcal{N}}_{A,r}$, we have  \begin{align}g(v)\tilde{E}_k(v)=\left\{\begin{aligned}
&0, & k<r,\\
&\binom{n}{r}^{-1}\prod_{i=r+1}^n \left(1+\kappa_i(v)^2\right)^{-1/2}, &k=r,\\
&\binom{n}{k}^{-1}\prod_{i=r+1}^n \left(1+\kappa_i(v)^2\right)^{-1/2}\\
&\qquad\cdot \sum\limits_{\substack{r+1\leq i_{r+1}<\cdots<i_k\leq n}}(\kappa_{i_{r+1}}(v)-1)\cdots(\kappa_{i_k}(v)-1), &r<k\leq n.\end{aligned}\right.\label{gtildeE}\end{align}

For any $v\in\tilde{\mathcal{N}}_{A,0}$, we have 
\begin{align}\label{jixian}
  &g(v)\tilde{E}_k(v)\nonumber\\=&
  \binom{n}{k}^{-1}\prod_{i=1}^n\left(1+\kappa_i(v)^2\right)^{-1/2}\sum\limits_{1\leq i_{1}<\cdots<i_k\leq n}(\kappa_{i_{1}}(v)-1)\cdots(\kappa_{i_k}(v)-1)\nonumber\\
  =&\binom{n}{k}^{-1}\prod_{i=1}^n\left(1+\kappa_i(v)^2\right)^{-1/2}\sum\limits_{\ell=0}\limits^{k}(-1)^{k-\ell}\binom{n-\ell}{k-\ell}\nonumber\\
  &\cdot\sum\limits_{1\leq i_{1}<\cdots<i_\ell\leq n}\kappa_{i_{1}}(v)\cdots\kappa_{i_\ell}(v)\nonumber\\
  =&\sum\limits_{\ell=0}\limits^{k}\left((-1)^{k-\ell}\binom{k}{\ell}\left(\prod_{i=1}^n\left(1+\kappa_i(v)^2\right)^{-1/2}\right.\right.\nonumber\\&
  \left.\left.\cdot\binom{n}{l}^{-1}\sum\limits_{1\leq i_{1}<\cdots<i_\ell\leq n}\kappa_{i_{1}}(v)\cdots\kappa_{i_\ell}(v)\right)\right)\nonumber\\
  =&\sum\limits_{\ell=0}\limits^k\binom{k}{\ell}g(v)E_k(v).\end{align}
  For any $r=1,\dots,n$ and $k> r$,  letting $\kappa_{1}, \cdots, \kappa_r$ in the second line and third line of \eqref{jixian} tend to $\infty$, we have
  \begin{align*}&g(v)\tilde{E}_k(v)\\
      =&\binom{n}{k}^{-1}\prod_{i=r+1}^n\left(1+\kappa_i(v)^2\right)^{-1/2}\\
      &\cdot\sum\limits_{r+1\leq i_{r+1}<\cdots<i_k\leq n}(\kappa_{i_{r+1}}(v)-1)\cdots(\kappa_{i_k}(v)-1)\\
      =&\sum\limits_{\ell=r+1}\limits^{k}\left((-1)^{k-\ell}\binom{k}{\ell}
      \left(\prod_{i=r+1}^n\left(1+\kappa_i(v)^2\right)^{-1/2}\right.\right.\\
      &\left.\left.\cdot\binom{n}{l}^{-1}\sum\limits_{{r+1}\leq i_{r+1}<\cdots<i_\ell\leq n}\kappa_{i_{r+1}}(v)\cdots\kappa_{i_\ell}(v)\right)\right)\\
      =&\sum\limits_{\ell=0}\limits^{k}(-1)^{k-\ell}\binom{k}{\ell}g(v)E_\ell(v), 
  \end{align*}
  for any  $v\in\tilde{\mathcal{N}}_{A,r}$.   After a straightforwardly verifying the remaining cases, we have
\begin{align}
    g(v)\tilde{E}_k(v)=\sum\limits_{\ell=0}\limits^{k}(-1)^{k-\ell}\binom{k}{\ell}g(v)E_\ell(v), \,\text{for any } v\in\tilde{\mathcal{N}}_{A}.\label{relationsbetweenge}
\end{align}
For any $k=0,\dots, n$,  the $\mathcal{L}^n$ measurability of $v\mapsto g(v)E_k(v)$ implies that the function $v\mapsto g(v)\tilde{E}_k(v)$ is also $\mathcal{L}^n$ measurable. 

 For any  $v\in \tilde{\mathcal{N}}_{A,0}$, by Kohlmann \cite{Kohlmann1991}, one can choose $$\left((u_1,\kappa_1(v) u_1),\dots, (u_n,\kappa_n(v)u_n) \right)$$ as a basis of $T_v\mathcal{N}_A$. Hence the Jacobian of 
\begin{align*}
    \Pi: \tilde{\mathcal N}_{A,0}&\rightarrow \partial A,\\
    v&\mapsto\Pi(v)
\end{align*}
is  $g(v)$. 
 Due to the convexity of $A$, we have $\mathcal{L}^n(\Pi(\tilde{\mathcal{N}}_{A,0}))=\mathcal{L}^n(\partial A)$. Noticing that $\mathcal{N}_A$ is an $n$-dimensional strong  Lipschitz manifold, by coarea formula in \cite{federe}  we have 
\begin{align}
 V_n(A,\beta)=&\int_{\mathcal{N}_A\cap \Pi^{-1}(\beta)}g(v)d\mathcal{L}^n v\nonumber\\
 =&\int_{\tilde{\mathcal{N}}_{A,0}\cap\Pi^{-1}(\beta)}   g(v)d\mathcal{L}^n v\nonumber\\
 =&\int_{\partial A\cap \beta}d\mathcal{L}^np=\mathcal{L}^n(\partial A\cap \beta).\label{Vn}
\end{align}

 When $A\in\mathcal{K}_{n+1}({\mathbb{H}^{n+1}})$ and $A$ is smooth, we denote by $N(p)$ the outer unit normal vector for any $p\in\partial A$. By a direct calculation we know that $g(v)$ actually equals the Jacobian of the map $N$ at $\Pi(v)$, from which we have  $$V_k(A, \beta)=\binom{n}{k}\int_{\partial A \cap \beta} E_{n-k}(\kappa) d \mu$$ for any $\beta\in\mathcal{B}(\mathbb{H}^{n+1})$.

Moreover, Kohlmann  proved  the following result in \cite[Lemma 2.4]{Kohlmann1994}.         \begin{thm}[\cite{Kohlmann1994}]\label{weakcon}
    Let $A, A_1, A_2, \dots$ be compact convex domains in $\mathbb{H}^{n+1}$ and \[\lim\limits_{i\rightarrow\infty}d_{\mathbb{H}^{n+1}}(A,A_i)=0.\] Then we have \begin{align*}
        V_k(A_i,\cdot)\stackrel{w}\rightharpoonup V_k(A,\cdot), \quad k=0,\dots,n.
    \end{align*}
\end{thm}

\subsection{Evolution equations}
In the following, we denote by $\Psi(\mathcal{W}-I)$ the curvature function  $\Psi(\mathcal{W}-I)=\tilde{E}_n$ for convenience, and denote by $\dot{\Psi}^{k l}(\mathcal{W}-I), \ddot{\Psi}^{k l, p q}(\mathcal{W}-I)$ the first and second derivatives of $\Psi$ with respect to the components of $\mathcal{W}-I$.

Assume that $M_t$ is a smooth solution to the flow (\ref{flow}) on the time interval $[0,T)$. Then we have the evolution equations of the induced metric,  Weingarten matrix, induced area element $d \mu_t$, and unit outer normal vector $\nu$ as follows (see \cite{AndrewsWei2018})
\begin{align}
    \frac{\partial}{\partial t} g_{i j} & =2(\phi(t)-\Psi(\mathcal{W}-I)) h_{i j}, \\ 
    \frac{\partial}{\partial t} \nu & =\nabla \Psi(\mathcal{W}-I), \\ 
    \frac{\partial}{\partial t} d \mu_{g_t} & =n E_1(\phi(t)-\Psi(\mathcal{W}-I)) d \mu_{g_t}, 
    \\ \frac{\partial}{\partial t} h_i^j & =\nabla^j \nabla_i \Psi(\mathcal{W}-I)+(\Psi(\mathcal{W}-I)-\phi(t))\left(h_i^k h_k^j-\delta_i^j\right),
    \end{align}
where $\delta^j_i$ equals $1$ when $i=j$ and $0$ otherwise. We may ignore the distinction between upper and lower indices on the right sides of the equations when deriving the evolution equations  because we can choose an orthonormal coordinate at a point.

Due to the shifted curvature term in flow (\ref{flow}), it is convenient to work with the shifted Weingarten matrix of $M_t$, which is denoted by $S=\mathcal{W}-I$. We express the evolution equation of $S$ along flow (\ref{flow}) in the following lemma. 
\begin{lemma}\label{lemmasij}
    \begin{align} \nonumber\frac{\partial}{\partial t} S^j_i=\frac{\partial}{\partial t} h^j_i=&\nabla^j \nabla_i \Psi(S)+(\Psi(S)-\phi(t))\left( (S^2)^j_i+2S^j_i\right) \\
    = & \dot{\Psi}^{k l}(S) \nabla_k \nabla_l S^j_i+\ddot{\Psi}^{k l, p q}(S) \nabla^j S^k_l \nabla_i S^p_q\nonumber\\
    &+((1-\alpha)\Psi(S)-\phi-\dot{\Psi}^{kl}(S)\delta^k_l)(S^2)^j_i\notag
    \\&+(\dot{\Psi}^{k l}(S^2)^k_l+2\Psi(S)-2\phi)S^j_i+\dot{\Psi}^{k l}(S^2)^k_l\delta^j_i.\label{Sij}\end{align}
\end{lemma}
\begin{proof}
     Using the identities $\dot{\Psi}^{kl}(S)S_{kl}=\alpha \Psi(S)$ and $(h^2)^j_i=(S^2)^j_i+2S^j_i+\delta^j_i$, the evolution equation of $S^j_i$ follows by a  calculation analogous to that in \cite[Lemma 2.4]{AndrewsWei2018}.
\end{proof}
Since \begin{align*}
        \frac{\partial}{\partial t}\Psi(S)=\dot{\Psi}^{kl}\frac{\partial}{\partial t}S^k_l,
    \end{align*} Lemma \ref{lemmasij} immediately yields the evolution equation of the curvature function $\Psi$.
\begin{lemma}
    \begin{align}\nonumber 
    \frac{\partial}{\partial t} \Psi(S)= & \dot{\Psi}^{k l}(S) \nabla_k \nabla_l \Psi(S)-\dot{\Psi}^{k l}(S) (\phi(t)-\Psi(S))(S^2)_{kl}\\
    &\quad -2\alpha\Psi(S)(\phi(t)-\Psi(S)).
    \end{align}
\end{lemma}

Denote by $\Omega_t$ the compact domain enclosed by $M_t$. From \cite{AndrewsChenWei2021}, we have the following lemma.
\begin{lemma}\label{quermass}
    \begin{align} \frac{\partial}{\partial t} \tilde W_l(\Omega_t)=\int_{M_t}E_l(S)(\phi(t)-\Psi(S))d\mu_{g_t}.\label{evolutionquer}\end{align}
\end{lemma}

By Lemma \ref{quermass}, we have $$\frac{\partial}{\partial t}  \mathrm{Vol}(\Omega_t)=\int_{M_t}(\phi(t)-\Psi(S))d\mu_{g_t}=0,$$
from which we know that the volume of $\Omega_t$ is preserved along flow (\ref{flow}).

\section{Modified curvature measures for h-convex domains}\label{rigidity}
In this section we derive some convergence properties of $\mathcal{H}_{n+1}({\mathbb{H}^{n+1}})$ and introduce the modified curvature measures of h-convex domains. Moreover, we adapt Kohlmann's  technique in \cite{Kohlmann1998} to characterize geodesic balls in $\mathbb{H}^{n+1}$ via modified curvature measures.

\subsection{Hausdorff convergence of h-convex domains}
In the following we consider the properties of the metric space $(\mathcal{H}_{n+1}(\mathbb{H}^{n+1}), d_{\mathbb{H}^{n+1}})$. It is known that a domain $\Omega$ in hyperbolic space is h-convex if and only if for any $x,y\in \Omega$($x\neq y$), every horocycle segment from $x$ to $y$ lies in $\Omega$. 

For $x,y\in\mathbb{H}^{n+1}$, denote by $H(x,y)$ the set of all points lying on some horocycle  segment from $x$ to $y$. To describe $H(x,y)$ more precisely,  we consider the Poincar\'e ball model of hyperbolic space $\mathbb{H}^{n+1}$ and take  $x$ to be the center $O$ of the unit ball. For any $y\neq O$, let $\Pi$ be a $2$-dimensional totally geodesic plane containing both $O$ and $y$. Viewed in the Poincar\'e ball model, $\Pi$ is a Euclidean disc obtained by intersecting a plane through $O$ with the unit ball. Then a horocycle segment from $O$ to $y$ in hyperbolic space corresponds to a minor arc $\overset{\large{\frown}}{Oy}$ of a circle in Euclidean space that lies in the $\overline{\Pi}$, passes through $O$ and $y$ and is tangent to the boundary of the  closure of $\Pi$. Then $H(O,y)$, viewed in Euclidean space, is  obtained by rotating the $\overset{\large{\frown}}{Oy}$ around the line $\overline{Oy}$ (see also in \cite{BorisenkoMiquel1999}).

Denote by $\mathcal{C}'(\mathbb{H}^{n+1})\subseteq \mathcal{C}(\mathbb{H}^{n+1})$ the collection of compact subsets of  $\mathbb{H}^{n+1}$ containing at least two distinct points.
  By the structure of $H(x,y)$ we explained above,  $H(x,y)$ is a compact subset of $\mathbb{H}^{n+1}$ and is continuous for each variable with respect to the Hausdorff metric. Then, by adapting the classical method for proving that  $\mathcal{K}(\mathbb{R}^{n+1})$ is closed (see \cite{schneiderbook}) to the h-convex setting, we  obtain the following lemma.
\begin{lemma}\label{propclosed}
    $\mathcal{H}_{n+1}(\mathbb{H}^{n+1})$ is a closed subset of $(\mathcal{C}'(\mathbb{H}^{n+1}),d_{\mathbb{H}^{n+1}})$.
\end{lemma}
\begin{proof}
    We will prove that $\mathcal{C}'(\mathbb{H}^{n+1})\setminus\mathcal{H}_{n+1}(\mathbb{H}^{n+1})$ is open.
    
    Let $A\in\mathcal{C}'(\mathbb{H}^{n+1})\setminus\mathcal{H}_{n+1}(\mathbb{H}^{n+1})$. In the following we prove that there exists a neighborhood of $A$ which is contained in $\mathcal{C}'(\mathbb{H}^{n+1})\setminus\mathcal{H}_{n+1}(\mathbb{H}^{n+1})$.  By $A\in\mathcal{C}'(\mathbb{H}^{n+1})\setminus\mathcal{H}_{n+1}(\mathbb{H}^{n+1})$,   there exists $x,y\in A, x\neq y$ such that $H(x,y)\cap A^c\neq \emptyset$. Choose $h\in H(x,y)\cap A^c$. Noting that $A^c$ is a open subset of $\mathbb{H}^{n+1}$, there exists $\varepsilon>0$ such that $B (h,\varepsilon)\subseteq A^c$. By the continuity of $H(x,y)$, there exists $\delta>0$ such that for any distinct $x',y'$ satisfying $d_{\mathbb{H}^{n+1}}(x',x)<\delta$ and $d_{\mathbb{H}^{n+1}}(y',y)<\delta$ we have $d_{\mathbb{H}^{n+1}}(H(x,y),H(x',y'))< \frac{\varepsilon}{3}$. Let $\delta'=\mathrm{min}\{\frac{\epsilon}{3},\delta, \frac{d_{\mathbb{H}^{n+1}}(x,y)}{2}\}$. Then for any $A'\in \mathcal{C}(\mathbb{H}^{n+1})$ satisfying $d_{\mathbb{H}^{n+1}}(A,A')< \delta'$, by the definition of Hausdorff metric, there exist $x',y'\in A'$ such that $d_{\mathbb{H}^{n+1}}(x',x)<\delta'$, $d_{\mathbb{H}^{n+1}}(y',y)<\delta'$ and $x'\neq y'$.  Besides we have $$d_{\mathbb{H}^{n+1}}(H(x,y),H(x',y'))< \frac{\varepsilon}{3},$$ from which we know that there exists $h'\in H(x',y')$ such that $d_{\mathbb{H}^{n+1}}(h,h')<\frac{\varepsilon}{3}$. Noticing that $$d_{\mathbb{H}^{n+1}}(A,A')< \frac{\varepsilon}{3},$$ we have $h'\in A'^c$. Hence $A'\notin \mathcal{H}_{n+1}(\mathbb{H}^{n+1})$. 
\end{proof}
By Lemma \ref{propclosed} we obtain the following Blaschke selection type theorem.
\begin{thm}
   For any bounded sequence $\{A_i, i\in\mathbb{N}\}\subset \mathcal{H}_{n+1}(\mathbb{H}^{n+1})$ whose elements  have a uniform positive lower bound on inner radius, there is a subsequence which converges to $A\in \mathcal{H}_{n+1}(\mathbb{H}^{n+1})$.\label{hausdorffofhconvex}
\end{thm}
\begin{proof}
    Suppose that $\{A_i, i\in\mathbb{N}\}\subseteq \mathcal{H}_{n+1}(\mathbb{H}^{n+1})$ is bounded whose elements have a uniform positive lower bound on inner radius. By the boundedness of $\{A_i, i\in\mathbb{N}\}$,  there exists a subsequence
    $\{A_{i_k}\}$ that converges to some $A\in \mathcal{C}(\mathbb{H}^{n+1})$. Due to each element of $\{A_{i_k}\}$ has a uniform lower bound on inner radius, we have $A\in\mathcal{C}'(\mathbb{H}^{n+1})$. By Lemma \ref{propclosed}, we have $A\in \mathcal{H}_{n+1}(\mathbb{H}^{n+1})$.
\end{proof}

\subsection{Generalized principal curvatures of h-convex domain}

 Next we show that, as in the smooth case, the generalized principal curvatures of the boundary of any h-convex domain are bounded below by $1$. 
 In Euclidean space, it is known that any parallel set of a convex bodies is convex. We prove that in hyperbolic space any parallel set of an h-convex body is h-convex.
\begin{lemma}
    Suppose that $A\in\mathcal{H}_{n+1}(\mathbb{H}^{n+1})$. Then for any $\varepsilon>0$, we have $A^\varepsilon\in \mathcal{H}_{n+1}(\mathbb{H}^{n+1})$. 
\end{lemma}
\begin{proof}
It is known that a compact domain $\Omega$ is h-convex if and only if for any $x\in\partial \Omega$ there exists a supporting horoball at $x$.  Suppose that $A\in\mathcal{H}_{n+1}(\mathbb{H}^{n+1})$. For any $x\in\partial A^\varepsilon$, we have $f_A(x)\in\partial A$. Let $B_1$ be the horoball such that $f_A(x)\in\partial B_1$ and  $F_A(x)$ is the unit outward normal of $\partial B_1$ at $f_A(x)$. We first prove that $A\subseteq B_1$. Now we work in the Poincar\'{e} ball model of $\mathbb{H}^{n+1}$ and take $f_A(x)$ to be the origin $O$ of the unit ball. Let $x'$ be the point obtained by intersecting the ray from $x$ to $f_A(x)$ with the boundary of the unit ball. Viewed in the Poincar\'{e} ball model, $B_1\cup x'$ is a Euclidean ball with the segment $f_A(x)x'$ as its diameter. Denote by $B_2$ the hyperbolic geodesic ball with $x$ as its center and $\varepsilon$ as its radius. Viewed in the Poincar\'{e} ball model, $B_2$ is  a Euclidean ball such that $x\in \mathring{B_2}$ and $B_2$ is externally tangent to $B_1$ at $f_A(x)$. If $A\not\subseteq B_1$, then there exists a point $y$ such that $y\in A$ and $y\not\in B_1$.  Then by the h-convexity of $A$, the horosegments set $H(f_A(x),y)\subseteq A$. By the geometric description of $H(f_A(x),y)$ and of horoballs in the Poincar\'{e} ball model, it follows that $H(f_A(x),y)\cap \mathring{B_2}\neq \emptyset$, which leads to a contradiction to the definition of $f_A(x)$.

%Denote by $r_2(t)$ the geodesic from $f_A(x)$ to $y$. Since $A$ is h-convex, we have $\overline{f_A(x)y}\subseteq A$. Moreover, there holds $\overline{f_A(x)y}\cap B=f_A(x)$. Then by the definition of $f_A(x)$, we have $ \langle \dot{r}(0), \dot{r}_2(0)\rangle_{\mathbb{H}^{n+1}}\geq 0$. Let $b_r$ be the Busemann function corresponding the geodesic $r$. By \cite{Heintze1977}, we have $\dot{b}_r(0)=-\dot{r}(0)=F_A(x)$. Hence 

It is known that  $B_1^\varepsilon$ is also a horoball. Moreover, since $B_1$ supports $A$ at $f_A(x)$ and $F_A(x)$ is a normal vector of $\partial B_1$ at $f_A(x)$, we can obtain that $B_1^\varepsilon$ supports $A^\varepsilon$ at $x$.  Hence we have $A^\varepsilon\in\mathcal{H}_{n+1}(\mathbb{H}^{n+1})$.
 \end{proof}
 
 In \cite{Walter1974}, Walter proved that $\partial A^\varepsilon$ is a $C^1$ hypersurface and $f_A$ and $F_A$ are twice differentiable almost everywhere on $\partial A^\varepsilon$.
For any $p\in \Phi_\varepsilon(\tilde{\mathcal{N}}_A)\subseteq \partial A^\varepsilon$, $\partial A^\varepsilon$ is twice differentiable at $p$. By a classical differential geometry argument,  the principal curvatures of $\partial A^\varepsilon$ are bounded below by $1$ at $p$. 
By \cite[Theorem 2.7]{Kohlmann1991} we have \begin{align*}
    \kappa_i(p)=\frac{\cosh(\varepsilon) \kappa_i(F_A(p))+ \sinh(\varepsilon)}{\sinh(\varepsilon) \kappa_i(F_A(p))+\cosh(\varepsilon)},
\end{align*}
As $\varepsilon$ tends to $0^+$,  we obtain the following lemma. \begin{lemma}
    \label{gepricuofh}
    For any $A\in\mathcal{H}_{n+1}(\mathbb{H}^{n+1})$ we have  $$\kappa_i(v)\geq 1, \quad \forall v\in\tilde{\mathcal{N}}_A.$$
\end{lemma}
Combining \eqref{relationsbetweenge} and Lemma \ref{gepricuofh}, we immediately obtain  that for any $k=0,\dots,n$, $\tilde{V}_k(A,\cdot)$ is a Borel measure on $\mathbb{H}^{n+1}$.

Now we consider the  generalized version of the Minkowski integral formulas for shifted principal curvatures. For any $q\in \mathbb{H}^{n+1}$, let $$S_q(v):=\sinh{\left(d_q(\Pi(v))\right)}\left\langle v, \operatorname{grad} d_q(\Pi(v))\right\rangle_{\mathbb{H}^{n+1}}, v\in TM, ||v||=1.$$ 
 The following theorem was proved by Kohlmann \cite[Theorem 3.4]{Kohlmann1994}.
\begin{thm}{\cite{Kohlmann1994}}\label{thmkohlmin}
    For any $A \in \mathcal{H}_{n+1}(\mathbb{H}^{n+1})$ and $r \in\{1, \ldots n\}$,  we have
$$\int_{\mathcal{N}_A} g(v)\left((\cosh d_q(\Pi(v)) E_{r-1}(v)-S_q(v) E_r(v)\right) \mathrm{d} \mathcal{L}^n v=0.$$
\end{thm}
Arguing as in \cite[Lemma 2.6]{HuLiWei2022} and combining the Theorem \ref{thmkohlmin} with \eqref{relationsbetweenge}, we obtain a modified  version of the  generalized Minkowski integral formulas for shifted principal curvatures.
\begin{thm}\label{shiftedminkow}
For any $A \in \mathcal{H}_{n+1}(\mathbb{H}^{n+1})$ and $r \in\{1, \ldots, n\}$,  we have
$$\int_{\mathcal{N}_A} g(v)\left((\cosh{\left(d_q(\Pi(v))-S_q(v)\right)} \tilde{E}_{r-1}(v)-S_q(v) \tilde{E}_r(v)\right) \mathrm{d} \mathcal{L}^n v=0.$$
\end{thm}
\subsection{Modified curvature measures for h-convex domains} In the following we  characterize  geodesic balls via modified curvature measures.
We adapt the technique of Kohlmann \cite{Kohlmann1998}  to prove Theorem \ref{modifiedcurvature}. The Lemma \ref{gepricuofh} and Theorem \ref{shiftedminkow} ensure that certain terms in the proof below are nonnegative.
\begin{proof}[proof of Theorem \ref{modifiedcurvature}]
Suppose that $k$ belongs to $\{2, \ldots, n\}$ and let $c>0$ be a constant. Suppose that $A\in\mathcal{H}_{n+1}(\mathbb{H}^{n+1})$ satisfying $$\tilde{V}_{n-k}(A, \beta)=c\tilde{V}_n(A, \beta),\quad \text{for any Borel set}\, \beta \in \mathcal{B}\left(\mathbb{H}^{n+1}\right),$$ i.e.,
\begin{align}
    \int_{\Pi^{-1}(\beta)\cap \mathcal{N}_A}g(v) \tilde{E}_k(v)d\mathcal{L}v=a^k\int_{\Pi^{-1}(\beta)\cap \mathcal{N}_A}g(v)d\mathcal{L}v,\label{tiaojian}
\end{align} 
here $a^k:=\binom{n}{k}^{-1}c$.
By Lemma \ref{gepricuofh}, we have \begin{align}
    \kappa_i(v)\geq 1,\quad \forall v\in\tilde{\mathcal{N}}_A, \,i=1,\dots, n\label{kia}
\end{align}
We first derive 
\begin{align}
   g(v), g(v)\tilde{E}_k(v)=0, \quad \text{a.e. in}\, \tilde{\mathcal{N}}_A\setminus\tilde{\mathcal{N}}_{A,0}.\label{gvek}
\end{align}
Proof of (\ref{gvek}):  
By the definition of $g$, we immediately obtain $g(v)=0$ for any $v\in\tilde{\mathcal{N}}_A\setminus\tilde{\mathcal{N}}_{A,0}$. If there exists a measurable set $B\subseteq \tilde{\mathcal{N}}_A\setminus\tilde{\mathcal{N}}_{A,0}$ such that $\mathcal{L}^n(B)>0$ and $g(v)\tilde{E}_k(v)>0$ on $B$. We have 
\begin{align}
    0<\int_{B}g(v) \tilde{E}_k(v)d\mathcal{L}^nv. \label{contra1}
\end{align}
By the convexity of $A$, we have that $\partial A$ is $\mathcal{L}^n$-almost everywhere twice differentiable, from which we have $\mathcal{L}^n(\Pi(B))=0$.  By (\ref{Vn}) and (\ref{contra1}), we have $$0=c\tilde{V}_n(A,\Pi(B))=\tilde{V}_{n-k}(A,\Pi(B))\geq\int_{B}g(v) \tilde{E}_k(v)d\mathcal{L}^nv> 0,$$ which leads to a contradiction. Hence we complete the proof of (\ref{gvek}).

We next prove 
\begin{align}
    \tilde{E}_k(v)=a^k \quad \text{for $\mathcal{L}^n$-almost every} \,v \in \tilde{\mathcal{N}}_{A,0}.\label{Ek}
\end{align}
Proof of (\ref{Ek}):
If there exists a measurable set $B\subseteq \tilde{\mathcal{N}}_{A,0}$ such that $\mathcal{L}^n(B)>0$ and $g(v)\tilde{E}_k(v)>a^k$ on $B$.
By argument in subsection \ref{plementaryofgpc}, the restriction $\Pi|_B$ is an injection and $g(v)>0$ on $B$. Hence by coarea formula in \cite{federe} we have 
\begin{align*}
    0=c\tilde{V}_n(A,\Pi(B))-\tilde{V}_{n-k}(A,\Pi(B))=\binom{n}{k}^{-1}\int_B g(v)(a^k-\tilde{E}_k(v))d\mathcal{L}^n v>0,
\end{align*} 
which leads to a contradiction.

By the definition of $g$, we have 
\begin{align}
    g(v)\tilde{E}_r(v)=0, \quad \text{for all} \,v\in \bigcup\limits_{i=r+1}\limits^{n}\tilde{\mathcal{N}}_{A,i}.\label{gver}
\end{align}
In the following, we  prove that $\kappa_i(v)=a+1$, $i=1,\dots,n$ for $\mathcal{L}^n$-almost every $v\in\tilde{\mathcal{N}}_{A,0}$ by considering the modified Minkowski formula in Theorem \ref{shiftedminkow}. 

Let $q\in \mathring{A}$. For all $\, v\in\mathcal{N}_A$, by \cite{Kohlmann1998} we have 
\begin{align}
    0< \sinh{\left(d_q(\Pi(v))\right)}\left\langle v, \operatorname{grad} d_q(\Pi(v))\right\rangle_{\mathbb{H}^{n+1}}<\cosh d_q(\Pi(v)).\label{positive}
\end{align}
 Let $\Gamma:=\bigcup_{i=1}^{k-1} \tilde{\mathcal{N}}_{A, i}$. For all $ v\in\tilde{\mathcal{N}}_A\setminus \Gamma$, we have 
\begin{align}
g(v)\tilde{E}_{k-1}(v)=0.\label{gamma}
\end{align} 

For any $r=\{1,\dots,n\}$, by Theorem \ref{shiftedminkow}, we have 
 \begin{align}
     &\int_{\mathcal{N}_A} g\left((\cosh\circ d_q \circ \Pi-S_q) \cdot \tilde E_{r-1}-S_q \tilde E_r\right) d\mathcal{L}^n v\nonumber\\
     =&\int_{\tilde{\mathcal{N}}_A} g\left((\cosh\circ d_q \circ \Pi-S_q) \cdot \tilde E_{r-1}-S_q \tilde E_r\right) d\mathcal{L}^n v\nonumber\\
     =&0.\label{regularminkow}
 \end{align}
Hence by \eqref{gvek} and \eqref{gver} we have 
\begin{align}
     &\int_{\tilde{\mathcal{N}}_{A,0}} g\left((\cosh\circ d_q \circ \Pi-S_q) \cdot \tilde E_{k-1}-S_q \tilde E_k\right) d\mathcal{L}^n v\nonumber\\
    =&-\int_{\tilde{\mathcal{N}}_{A}\setminus\tilde{\mathcal{N}}_{A,0}} g\left((\cosh\circ d_q \circ \Pi-S_q) \cdot \tilde E_{k-1}-S_q \tilde E_k\right) d\mathcal{L}^n v\nonumber\\
    =&-\int_{\tilde{\mathcal{N}}_{A}\setminus\tilde{\mathcal{N}}_{A,0}} g\cdot(\cosh\circ d_q \circ \Pi-S_q) \cdot \tilde E_{k-1} d\mathcal{L}^n v\nonumber\\
    =&-\int_{\Gamma} g\cdot(\cosh\circ d_q \circ \Pi-S_q) \cdot \tilde E_{k-1} d\mathcal{L}^n v\nonumber\\
    \leq&0, \label{minkowEk}
 \end{align}
 where the second equality holds by (\ref{gvek}), the third equality holds by (\ref{gamma}) and the last equality holds by (\ref{kia}). Similarly we have 
 
\begin{align}
    &\int_{\tilde{\mathcal{N}}_{A,0}} g\left(S_q \tilde E_1-(\cosh\circ d_q \circ \Pi-S_q)\right) d\mathcal{L}^n v\nonumber\\
    =&-\int_{\tilde{\mathcal{N}}_{A}\setminus\tilde{\mathcal{N}}_{A,0}} g\left(S_q \tilde E_1-(\cosh\circ d_q \circ \Pi-S_q) \right) d\mathcal{L}^n v\nonumber\\
    =&-\int_{\tilde{\mathcal{N}}_{A}\setminus\tilde{\mathcal{N}}_{A,0}} gS_q \tilde E_1 d\mathcal{L}^n v\nonumber\\
    =&-\int_{\tilde{\mathcal{N}}_{A,1}} gS_q \tilde E_1 d\mathcal{L}^n v\nonumber\\
    \leq&0, \label{minkowE1}
\end{align}
where the second equality holds by (\ref{gvek}) and the last equality holds by (\ref{kia}) and (\ref{positive}).

Next we will apply the  inequality and its equality conditions derived by Simon \cite{simon}: \begin{align}
   & E_s^{1/ s}(v)-E_t^{1/ t}(v) \geq0, \quad\text{for}\, 1\leq s<t\leq n,\label{simon}
\end{align}
where equality holds iff  $\kappa_1(v)=\ldots=\kappa_n(v)$  or  $H_s(v)=0$.

By \eqref{kia}, \eqref{Ek}, \eqref{positive}, \eqref{minkowEk}, \eqref{minkowE1} and \eqref{simon}, we obtain
\begin{align}
    \nonumber0\leq& \int_{\tilde{\mathcal{N}}_{A,0}} g(\cosh \circ d_q \circ \Pi-S_q)(\tilde{E}_{k-1}-\tilde{E}_k^\frac{k-1}{k})a^{1-k} d\mathcal{L}^n v
    \\=& \int_{\tilde{\mathcal{N}}_{A,0}} g(\cosh \circ d_q \circ \Pi-S_q)(\tilde{E}_{k-1}-\tilde{E}_k^\frac{k-1}{k})\tilde{E}_k^\frac{1-k}{k}d\mathcal{L}^n v\nonumber\\ \nonumber
    =&\int_{\tilde{\mathcal{N}}_{A,0}}g\left((\tilde E_k^\frac{1}{k}-\tilde E_1)S_q+((\cosh \circ d_q \circ \Pi-S_q)\tilde{E}_{k-1}-S_q\tilde{E}_k)\tilde{E}_k^\frac{1-k}{k}\right.\\ \nonumber
    &\left.+(S_q\tilde{E}_1-(\cosh \circ d_q \circ \Pi-S_q)          \right)d\mathcal{L}^n v\\ \nonumber
    =&\int_{\tilde{\mathcal{N}}_{A,0}}g\left((\tilde E_k^\frac{1}{k}-\tilde E_1)S_q+((\cosh \circ d_q \circ \Pi-S_q)\tilde{E}_{k-1}-S_q\tilde{E}_k)a^{1-k}\right.\\ \nonumber
    &\left.+(S_q\tilde{E}_1-(\cosh \circ d_q \circ \Pi-S_q)          \right)d\mathcal{L}^n v\nonumber
    \\
    \leq& 0, \label{kieq1} 
\end{align}
from which it follows that \begin{align}
    \int_{\tilde{\mathcal{N}}_{A,0}}g(\tilde E_k^\frac{1}{k}-\tilde E_1)S_q d\mathcal{L}^n v=0.\label{kieq2}
\end{align}
Noticing that $g(v)>0$ for $v\in\tilde{\mathcal{N}}_{A,0}$ and combining (\ref{positive}), we have 
$\tilde E_k^\frac{1}{k}-\tilde E_1=0$ for $\mathcal{L}^n$-almost every $v\in\tilde{\mathcal{N}}_{A,0}$. Then, by (\ref{Ek}), for $\mathcal{L}^n$-almost every $v\in\tilde{\mathcal{N}}_{A,0}$, we have \begin{align}
    \kappa_i(v)=a+1,\quad i=1,\dots,n.\label{kieqa}
\end{align}

We next prove 
\begin{align}
 \mathcal{L}^n(\cup_{i=1}^{r}\tilde{\mathcal{N}}_{A,i})=0,\quad r=1,\dots,n.\label{regular}   
\end{align}
Proof of (\ref{regular}): 
 We first note that 
\[g(v) \tilde E_r(v)=\binom{n}{r}^{-1} \prod_{i=r+1}^n\left(\kappa_i(v)^2+1\right)^{-1/2}>0\] 
for $v \in \tilde{\mathcal{N}}_{A, r}$. We prove (\ref{regular}) by induction. For $r=1$, by (\ref{kieq1}) we have 
\begin{align*}
    \int_{\tilde{\mathcal{N}}_{A,0}}g(S_q\tilde{E}_1-(\cosh \circ d_q \circ \Pi-S_q))d\mathcal{L}^nv=0.
\end{align*}
Hence by \eqref{minkowE1} we have \begin{align*}
    \int_{\tilde{\mathcal{N}}_{A,1}} gS_q \tilde E_1 d\mathcal{L}^n v=0.
\end{align*}
Noticing that $g(v)\tilde{E}_1(v)>0$ and $S_q>0$ on $\tilde{\mathcal{N}}_{A,1}$, we have 
$\mathcal{L}^n(\tilde{\mathcal{N}}_{A,1})=0$.  Now if we have obtained $\mathcal{L}^n(\cup_{i=1}^{r}\tilde{\mathcal{N}}_{A,i})=0$ for a fixed $1\leq r\leq n-1$. Then, by
 \eqref{gtildeE}, Theorem \ref{shiftedminkow}, \eqref{kieqa} and  $\mathcal{L}^n(\cup_{i=1}^{r}\tilde{\mathcal{N}}_{A,i})=0$,  we    have \begin{align*}
     & \int_{\tilde{\mathcal{N}}_{A, r+1}}g S_q  \tilde{E}_{r+1} d \mathcal{L}^n v \\ 
     =& \int_{\tilde{\mathcal{N}}_{A, r+1}}g \left(S_q  \tilde{E}_{r+1}-(\cosh\circ d_q\circ\Pi-S_q)\tilde{E}_r\right) d \mathcal{L}^n v\\
     =&\int_{\tilde{\mathcal{N}}_{A}\setminus\tilde{\mathcal{N}}_{A, 0}}g \left(S_q  \tilde{E}_{r+1}-(\cosh\circ d_q\circ\Pi-S_q)\tilde{E}_r\right) \mathcal{L}^n v\\
     =&\int_{\tilde{\mathcal{N}}_{A,0}} g\left((\cosh \circ d_q\circ\Pi-S_q) \tilde{E}_r-S_q \tilde{E}_{r+1}\right) d \mathcal{L}^n v\\
     =&\int_{\tilde{\mathcal{N}}_{A,0}} g\left((\cosh \circ d_q\circ\Pi-S_q) a^r-S_q a^{r+1}\right) d \mathcal{L}^n \\=&a^r \int_{\tilde{\mathcal{N}}_{A,0}} g\left((\cosh \circ d_q\circ\Pi-S_q)-S_q \tilde{E}_1\right) d \mathcal{L}^n\\
     =&a^r\int_{\tilde{\mathcal{N}}_{A,1}} gS_q \tilde E_1 d\mathcal{L}^n v=0,
\end{align*} 
from which we have $\mathcal{L}^n(\tilde{\mathcal{N}}_{A,r+1})=0$. We complete the proof of (\ref{regular}).

Now, by (\ref{kieqa}) and (\ref{regular}), we have 
\begin{align*}
    V_{n-k}(A,\beta)&=\binom{n}{k}\int_{\Pi^{-1}({\beta})\cap\mathcal{N}_A} g E_k d\mathcal{L}^n v\\
    &=\binom{n}{k}\int_{\Pi^{-1}({\beta})\cap\tilde{\mathcal{N}}_{A,0}} g E_k d\mathcal{L}^n v\\
    &=\binom{n}{k}(a+1)^k\int_{\Pi^{-1}({\beta})\cap\tilde{\mathcal{N}}_{A,0}} g d\mathcal{L}^n v\\
    &=\binom{n}{k}(a+1)^kV_n(A,\beta), \quad \text{for any} \, \beta\in\mathcal{B}(\mathbb{H}^{n+1}).
\end{align*}
Then, by Theorem \ref{kohlmann98thm}, it follows that $A$ is a geodesic ball.
\end{proof}
\subsection{Relations between modified quermassintegrals and modified curvature integrals}
There are some identities relating modified quermassintegrals and modified curvature integrals for smooth domains. The following lemma was stated in \cite{Li-Xu}. For the convenience of the reader we give a calculation in this subsection. 
\begin{lemma}\label{lemma3.4}
    Let $A\in\mathcal{H}_{n+1}(\mathbb{H}^{n+1})$ be a smooth domain. 
    For  $1\leq k\leq n-1$, we have
    \begin{align}
        \int_{\partial A} \tilde{E}_k d\mathcal{L}^n p=(n-2k)\tilde{W}_k(A)+(n-k)\tilde{W}_{k+1}(A).\label{ekquer}
    \end{align}
Moreover, when $k=n$ we have
    \begin{align}
        \int_{\partial A} \tilde{E}_n d\mathcal{L}^n p=\omega_n-n\tilde{W}_n(A).\label{modifiedcurvaturemeasure}
    \end{align}
\end{lemma}

We first derive a combinatorial lemma.  Set $W_{-1}(A)=0$ for all $A\in\mathcal{H}_{n+1}(\mathbb{H}^{n+1})$.
\begin{lemma}
    For any smooth domain $A\in\mathcal{H}_{n+1}(\mathbb{H}^{n+1})$, there holds
    \begin{align}
        &\sum\limits^{k}\limits_{i=0}(-1)^{k-i}\binom{k}{i}\left((n-i)W_{i+1}(A)+iW_{i-1}(A)\right)\nonumber\\=&(n-2k)\tilde{W}_k(A)+(n-k)\tilde{W}_{k+1}.\label{comb}
    \end{align}
\end{lemma}
\begin{proof}
          By a direct calculation, 
      \begin{align}
        &\sum\limits^{k}\limits_{i=0}(-1)^{k-i}\binom{k}{i}\left((n-i)W_{i+1}(A)+iW_{i-1}(A)\right)\nonumber\\=&\sum\limits^{k}\limits_{i=0}(-1)^{k-i}\binom{k}{i}(n-i)W_{i+1}(A)+\sum\limits^{k}\limits_{i=0}(-1)^{k-i}\binom{k}{i}iW_{i-1}(A)\nonumber\\
        =&n\sum\limits^{k}\limits_{i=0}(-1)^{k-i}\binom{k}{i}W_{i+1}(A)-\sum\limits^{k}\limits_{i=0}(-1)^{k-i}\binom{k}{i}iW_{i+1}(A)\nonumber\\
        &\quad + \sum\limits^{k}\limits_{i=0}(-1)^{k-i}\binom{k}{i}iW_{i-1}(A).\label{combi}
    \end{align}
    Using $\binom{k}{i}=\binom{k+1}{i+1}-\binom{k}{i+1}$ together with (\ref{Quermassine}), we have
    \begin{align}
        &n\sum\limits^{k}\limits_{i=0}(-1)^{k-i}\binom{k}{i}W_{i+1}(A)\nonumber\\=&n\sum\limits^{k}\limits_{i=0}(-1)^{k-i}\left(\binom{k+1}{i+1}-\binom{k}{i+1}\right)W_{i+1}(A)\nonumber\\
        =&n\tilde{W}_{k+1}(A)+n\tilde{W}_k(A).\label{lemma3.31}
    \end{align}
    Noticing that $$\binom{k-1}{i-1}=\binom{k}{i}-\binom{k-1}{i}=\binom{k+1}{i+1}-2\binom{k}{i+1}+\binom{k-1}{i+1}$$ and combining (\ref{Quermassine}), a similar calculation yields 
    \begin{align}
        &\sum\limits^{k}\limits_{i=0}(-1)^{k-i}\binom{k}{i}iW_{i+1}(A)\nonumber\\=&k\sum\limits^{k}\limits_{i=0}(-1)^{k-i}\binom{k-1}{i-1}W_{i+1}(A)\nonumber\\
        =&k\sum\limits^{k}\limits_{i=0}(-1)^{k-i}\left(\binom{k+1}{i+1}-2\binom{k}{i+1}+\binom{k-1}{i+1}\right)W_{i+1}(A)\nonumber\\
        =&k\tilde{W}_{k+1}(A)+2k\tilde{W}_{k}(A)+k\tilde{W}_{k-1}(A)\label{lemma3.32}
        \end{align}

        Similarly we have \begin{align}
            \sum\limits^{k}\limits_{i=0}(-1)^{k-i}\binom{k}{i}iW_{i-1}(A)=k\sum\limits^{k}\limits_{i=0}(-1)^{k-i}\binom{k-1}{i-1}W_{i-1}(A)=k\tilde{W}_{k-1}(A).\label{lemma3.33}
        \end{align}
        Plugging (\ref{lemma3.31}),(\ref{lemma3.32}) and (\ref{lemma3.33}) into (\ref{combi}), we obtain (\ref{comb}).
\end{proof}

Now we derive the relations between modified quermassintegrals and modified curvature integrals for smooth h-convex domains.

\begin{proof}[proof of Lemma \ref{lemma3.4}]
For $1\leq k\leq n-1$, by (\ref{QandC}) and (\ref{comb}), we have
\begin{align*}
    \int_{\partial A} \tilde{E}_k d\mathcal{L}^n p&=\sum\limits_{i=0}\limits^{k} (-1)^{k-i}\binom{k}{i}\int_{\partial A} E_i d\mathcal{L}^n p\\&=(-1)^k|\partial A|+\sum\limits_{i=1}\limits^{k} (-1)^{k-i}\binom{k}{i}\left((n-i)W_{i+1}(A)+iW_{i-1}(A)\right)\\&=(-1)^k|\partial A|+\sum\limits_{i=0}\limits^{k} (-1)^{k-i}\binom{k}{i}\left((n-i)W_{i+1}(A)+iW_{i-1}(A)\right)\\
    &\qquad -(-1)^k nW_1(A)\\
    &=(n-2k)\tilde{W}_k(A)+(n-k)\tilde{W}_{k+1}(A).
\end{align*}

For $k=n$, we have \begin{align}
        \int_{\partial A} \tilde{E}_n d\mathcal{L}^n p=\sum\limits_{i=0}\limits^{n} (-1)^{n-i}\binom{n}{i}\int_{\partial A} E_i d\mathcal{L}^n p. \label{shiftedE}
    \end{align}
Combining (\ref{QandC}) and (\ref{Quermassine}), we have
\begin{align*}
    \int_{\partial A} \tilde{E}_n d\mathcal{L}^n p=&\sum\limits_{i=0}\limits^{n} (-1)^{n-i}\binom{n}{i}\int_{\partial A} \tilde{E}_i d\mathcal{L}^n p\\
    =&nW_1(A)+\omega_n+nW_{n-1}(A)\\&+\sum\limits_{i=1}\limits^{n-1} (-1)^{n-i}\binom{n}{i}\left((n-i)W_{i+1}(A)+iW_{i-1}(A)\right)\\
    =&nW_1(A)+\omega_n+nW_{n-1}(A)\\&+n\sum\limits_{i=1}\limits^{n-1} (-1)^{n-i}\left(\binom{n-1}{i}W_{i+1}(A)+\binom{n-1}{i-1}W_{i-1}(A)\right)\\
    =&nW_1(A)+\omega_n+nW_{n-1}(A)\\
    &+n\sum\limits_{i=2}\limits^{n} (-1)^{n-i+1}\binom{n-1}{i-1}W_{i}(A))\\
    &+n\sum\limits_{i=0}\limits^{n-2} (-1)^{n-i+1}\binom{n-1}{i}W_{i}(A)).
\end{align*}
Since $\binom{n-1}{i-1}+\binom{n-1}{i}=\binom{n}{i}$, we have
\begin{align*}
\int_{\partial A} \tilde{E}_n d\mathcal{L}^n p=&\omega_n-n\tilde{W}_n(A).
\end{align*}
\end{proof}

%By Lemma \ref{lemma3.4},  the following corollary follows. 
%\begin{corollary}
 %   For any smooth domain $A\in\mathcal{H}_{n+1}(\mathbb{H}^{n+1})$ there holds 
 %   \begin{align*}
  %  (n-2k)\tilde{W}_k(A)+(n-k)\tilde{W}_{k+1}&\geq 0, k=1,\dots,n-1,
  %  \\ \omega_n-n\tilde{W}_n(A)&\geq 0.
  %  \end{align*}
%\end{corollary}

\section{Shifted Gauss curvature flows}\label{SGCF}

\subsection {A priori estimates and long time existence}\label{prio}
  
In this subsection, we derive the uniform upper and lower bounds on the inner radius and outer radius of the domains enclosed by the solution to the flow (\ref{flow}) and prove that the strict h-convexity is preserved along the flow, from which we can obtain the time-dependent upper and lower bounds on the shifted principal curvatures. Hence by a standard PDE argument we can obtain the long time existence of solutions to  the flow (\ref{flow}) with strictly h-convex initial hypersurfaces.

 Let $M_t=\partial \Omega_t$ be a smooth  solution of the flow (\ref{flow}) on the time interval $[0, T)$, with $M_0$  strictly h-convex.
Since  the volume of $\Omega_t$ is preserved along the flow (\ref{flow}), we obtain the following lemma.
\begin{lemma} \label{C0estimate}
   If $M_t$ is strictly h-convex for all $t\in[0,T)$, then there exist constants $c_1, c_2>0$, depending on $n$ and $M_0$, such that $$0<c_1<\rho_{-}(\Omega_t)<\rho_{+}(\Omega_t)<c_2.$$
 \end{lemma}
\begin{proof}
    Since the enclosed volume is preserved, we have $$\mathrm{Vol}( B(p_t,\rho_- (\Omega_t))\leq \mathrm{Vol}(\Omega_t)=\mathrm{Vol}(\Omega_0)\leq \mathrm{Vol}( B(p'_t,\rho_+(\Omega_t)),$$ where $B(p_t,\rho_- (\Omega_t))$ and $B(p'_t,\rho_+(\Omega_t)$  are respectively the inner ball and the outer ball of $\Omega_t$.  This, together with  (\ref{innerandouter}),  completes the proof.
\end{proof}

Let $M_t=\partial \Omega_t$ be a smooth strictly h-convex solution of the flow (\ref{flow}) on $[0,T)$.  By Lemma \ref{C0estimate}, for any $t\in[0,T)$ we can choose an inner ball $B(p_t,\rho_-(\Omega_t))\subseteq\Omega_t$. The following lemma shows that  for any $t\in [0,T)$ one can find a geodesic ball that remains inside $\Omega_t$ for a short subsequent time interval. 

\begin{lemma}\label{constrainedball}
        Let $M_t=\partial \Omega_t$ be a smooth strictly h-convex solution to the flow (\ref{flow}) on the time interval $[0,T)$. For any $t_0\in[0,T)$,  there exists a constant $\tau>0$ depending on $M_0$, $n$ and $\alpha$  such that 
        $$B\left(p_0, \frac{\rho_- (\Omega_{t_0})}{2}\right) \subseteq \Omega_t, \,\text{for all} ~\,t \in\left[t_0, \min \left\{T, t_0+\tau\right\}\right),$$ where $B(p_0, \rho_- ({\Omega}_{t_0}))$ is an inner ball of $\Omega_{t_0}$.
    \end{lemma}
\begin{proof}
    Consider the geodesic sphere centered at $p_0$ whose radius evolves under the ODE 
$$\left\{\begin{array}{l}r_B^{\prime}(t)=-\Psi\left((\operatorname{coth}r_B-1)I\right), \\ r_B(t_0)=\frac{3\rho_- (\Omega_{t_0})}{4}.\end{array}\right.$$

Let 
\[\tau_1=\int_{\rho_- (\Omega_{t_0}) /2}^{\frac{3\rho_- (\Omega_{t_0})}{4}} \frac{1}{\Psi\left((\operatorname{coth}(s)-1)I\right)}d s,\] 
at which time the radius of geodesic ball decreases from $\frac{3\rho_- (\Omega_{t_0})}{4}$ to $\frac{\rho_- (\Omega_{t_0})}{2}$. Noticing that $\rho_- (\Omega_{t_0})$ is bounded by Lemma \ref{C0estimate}, then $\tau_1$ has a lower bound $\tau>0$ depending on $M_0$ and $n$. We claim that  $\tau$ is exactly the one we require in this lemma: If it is not true, let $f(x,t)=d_{p_0}(x,t)-r_B(t)$.  Since  $f(x,t_0)>0$, there exists $t^*\in[t_0,\min\{t_0+\tau,T\})$ such that $f(x,t^*)$  attains $0$ for the first time at some point $p\in M$.  Then we have $$\frac{\partial}{\partial t}f(p,t^*)\leq 0.$$ However, by (\ref{flow}) we have \begin{align}\label{partialtd}
    \frac{\partial}{\partial t}d_{p_0}(x,t)=\langle\partial d_{p_0},\frac{\partial}{\partial t}X(x,t)\rangle_{\mathbb{H}^{n+1}}=(\phi(t)-\Psi(S))\langle\partial d_{p_0},\nu\rangle_{\mathbb{H}^{n+1}}> -\Psi(S).
\end{align} Then we can obtain a contradiction since $$\frac{\partial}{\partial t}f(p,t^*)>-\Psi(S)+\Psi((\operatorname{coth}((d_{p_0}(p))-1)I)\geq 0.$$ 
\end{proof}
Suppose that  $M_t$ is a strictly h-convex solution to the flow \eqref{flow}. For any $t_0\in[0,T)$, let $B(p_0, \rho_-(\Omega_{t_0}))$ be an inner ball of $\Omega_{t_0}$. 

Let  $u(x,t):= \operatorname{sinh}d_{p_0}(x)\langle \partial d_{p_0}, \nu\rangle_{\mathbb{H}^{n+1}}$ be the support function  of $M_t$ with respect to $p_0$. Combining (\ref{support}) and Lemma \ref{constrainedball}, there exists $c_3>0$ depending on $n$, $M_0$ and $\alpha$ such that \begin{align}u>2c_3, \,\text{for any} \,t\in \left[t_0, \min \left\{T, t_0+\tau\right\}\right). \label{ulow}\end{align} 

In the following we prove that the strict h-convexity is preserved along the flow (\ref{flow}).  We first derive the evolution equation of the support function.
\begin{lemma}
 \begin{align}\frac{\partial}{\partial t} u=&\dot{\Psi}^{k l}(S) \nabla_k \nabla_l u+\cosh d_{p_0}(x)\left(\phi(t)-(1+\alpha)\Psi(S)-\dot{\Psi}^{k l}(S) \delta^k_l\right)\nonumber\\
 &+(\dot{\Psi}^{kl}(S)(S^2)_{k l}+2\alpha\Psi(S)+\dot{\Psi}^{kl}(S)\delta^k_l) u.\label{u}\end{align}
\end{lemma}
\begin{proof}
   As in \cite{GuanLi2015} we know that $\sinh{d_{p_0}}\partial d_{p_0}$ is a conformal Killing field, i.e., \begin{align}
       \left\langle\bar{\nabla}_X\left(\sinh d_{p_0} \partial d_{p_0}\right), Y\right\rangle_{\mathbb{H}^{n+1}}=\cosh d_{p_0}\langle X, Y\rangle_{\mathbb{H}^{n+1}}\label{killing}
   \end{align} for any vector fields $X,Y$ on $\mathbb{H}^{n+1}$, by which we know that the spatial derivatives of $u$ are as follows: \begin{align}
        \nabla_k u & =\sinh d_{p_0}(x)\left\langle\partial d_{p_0}, h_k^i \partial_i\right\rangle_{\mathbb{H}^{n+1}} \label{firstdivu} \\ \nabla_k \nabla_l u & =\cosh d_{p_0}(x) h_{k l}+\sinh d_{p_0}(x)\left\langle\partial d_{p_0}, \nabla^j h_{k l} \partial_j-h_k^j h_{jl} \nu\right\rangle_{\mathbb{H}^{n+1}}.\label{seconddivu} 
   \end{align} 
   Combining (\ref{killing}), (\ref{firstdivu}) and (\ref{seconddivu}), we have 
   \begin{align*}
       \frac{\partial}{\partial t}u=&\frac{\partial}{\partial t} \langle \sinh d_{p_0}\partial d_{p_0}, \nu\rangle_{\mathbb{H}^{n+1}}=\cosh d_{p_0}\langle\frac{\partial}{\partial t} X, \nu\rangle_{\mathbb{H}^{n+1}}+\langle\sinh d_{p_0} \partial d_{p_0}, \frac{\partial}{\partial t}\nu\rangle_{\mathbb{H}^{n+1}}\\
       =&\cosh d_{p_0} (\phi(t)-\Psi(S))+\langle\sinh d_{p_0} \partial d_{p_0}, \nabla \Psi(S)\rangle_{\mathbb{H}^{n+1}}\\
       =&\cosh d_{p_0} (\phi(t)-\Psi(S))+\langle\sinh d_{p_0} \partial d_{p_0}, \dot{\Psi}^{kl}\nabla^j h_{kl}\partial_j\rangle_{\mathbb{H}^{n+1}}\\
       =&\dot{\Psi}^{k l}(S) \nabla_k \nabla_l u+\cosh d_{p_0}(x)\left(\phi(t)-\Psi(S)-\dot{\Psi}^{k l}(S) h_{k l}\right)\\&+\dot{\Psi}^{kl}(S) (h^2)_{k l} u\\
       =&\dot{\Psi}^{k l}(S) \nabla_k \nabla_l u+\cosh d_{p_0}(x)\left(\phi(t)-(1+\alpha)\Psi(S)-\dot{\Psi}^{k l}(S) \delta^k_l\right)\\&+(\dot{\Psi}^{kl}(S)(S^2)_{k l}+2\alpha\Psi(S)+\dot{\Psi}^{kl}(S)\delta^k_l) u.
   \end{align*}
\end{proof}
Similarly to \cite{WeiYangZhou2025}, we obtain the evolution equation of $d_{p_0}$.
\begin{lemma}
    \begin{align}
 \frac{\partial}{\partial t}d_{p_0}=&\dot{\Psi}^{kl}(S)\nabla_k\nabla_ld_{p_0}+\coth d_{p_0}\dot{\Psi}^{kl}(S)\nabla_kd_{p_0}\nabla_ld_{p_0}+\langle\partial d_{p_0}, \nu\rangle_{\mathbb{H}^{n+1}}\phi(t)\nonumber\\&+\left(\frac{\alpha}{\sinh d_{p_0}}u-\langle \partial d_{p_0},\nu\rangle_{\mathbb{H}^{n+1}}\right)\Psi(S)\nonumber\\&+\left(\frac{u}{\sinh d_{p_0}}-\coth d_{p_0}\right)\dot{\Psi}^{kl}(S)\delta^k_l.\label{dp0}\end{align}
\end{lemma}
\begin{proof}
    By \cite[Lemma 2.2]{WeiYangZhou2025}, we have \begin{align}\label{hessd}
     \nabla_j \nabla_i d_{p_0}&=\operatorname{coth} d_{p_0}\left(g_{i j}-\nabla_i d_{p_0} \nabla_j d_{p_0}\right)-\frac{u h_{i j}}{\sinh d_{p_0}}\\
     &=\operatorname{coth} d_{p_0}\left(g_{i j}-\nabla_i d_{p_0} \nabla_j d_{p_0}\right)-\frac{u (S_{i j}+g_{ij})}{\sinh d_{p_0}}.
    \end{align}
    Combining \eqref{partialtd} and \eqref{hessd}, we obtain \eqref{dp0}.
\end{proof}
Now we prove that the shifted principal curvatures of solution $M_t$ to the flow \eqref{flow} have a lower bound depending on $n$, $M_0$, $\alpha$ and $T$ on finite time interval $[0,T)$.
\begin{thm}\label{thmlower}
    Let $M_t=\partial\Omega_t$ be a smooth solution to flow (\ref{flow})  on  $[0,T)$. If $M_0$ is strictly h-convex, then there exists $C_1(T)$ depending on $n$, $M_0$, $\alpha$ and $T$ such that the shifted principal curvatures of $M_t$ satisfy $$\frac{1}{C_1(T)}\leq \kappa_i(p)-1, \,\text{for all}\ p\in M_t,\ i=1, \dots, n.$$ 
\end{thm}
\begin{proof}
     Firstly, we remark that it suffices to prove this theorem under the assumption that $M_t$ is strictly h-convex along the flow (\ref{flow}). If the theorem holds under the strict h-convexity assumption, then by the smoothness of the solution we can denote by $[0,T^*)\subseteq[0,T)$ the maximal time interval on which $M_t$ preserves strict h-convexity. Then a standard extension argument yields $T^*=T$.

Now we assume that $M_t$ is a strictly h-convex solution to the flow \eqref{flow}. Denote by $B:=\{B^j_i\}$ the inverse matrix of $S:=W-I$. Once we obtain an upper bound on $B$ depending on $n$, $M_0$, $\alpha$ and $T$, the proof is complete. We first derive the evolution equation of $B$. We have \begin{align}
    \frac{\partial}{\partial t}B^j_i&=-B^r_iB^j_s\frac{\partial}{\partial t}S^s_r,\label{ptbij}\\
    \nabla_l B^j_i&=-B^r_iB^j_s\nabla_lS^s_r,\label{yicibij}\\
    \nabla_k \nabla_l B_i^j&=-B_s^j B_i^r \nabla_k \nabla_l S^s_r-B_s^j \nabla_k B_i^r \nabla_l S^s_r-B^r_i\nabla_kB^j_s\nabla_l S^s_r.\label{ercibij}
\end{align}
By (\ref{ptbij}), (\ref{yicibij}), (\ref{ercibij}) and (\ref{Sij}), we have \begin{align}
\frac{\partial}{\partial t} B^j_i=&-B^r_iB^j_s\frac{\partial}{\partial t}S^s_r\nonumber\\
=&\dot{\Psi}^{k l}(S)\nabla_k\nabla_l B^j_i+2\dot{\Psi}^{k l}(S)B^q_i\nabla_k S^p_q \nabla_l B^j_p-\ddot{\Psi}^{kl,pq}B^j_s B^r_i\nabla_r S^k_l \nabla^s S^p_q\nonumber\\
&-\dot{\Psi}^{k l}(S)(S^2)_{kl} (B^2)^j_i-(\dot{\Psi}^{k l}(S)(S^2)_{kl}+2\Psi(S)-2\phi)B^j_i\nonumber\\
&-((1-\alpha)\Psi(S)-\phi-\dot{\Psi}^{k l}(S)\delta^k_l)\delta^j_i.\label{Bij}
\end{align}
For any $t_0\in[0,T)$, let $p_0$ be a center of inner ball of $M_{t_0}$.
Let $\lambda=\cosh d_{p_0}$. By \eqref{dp0} we have \begin{align}
    \frac{\partial}{\partial t}\lambda=\dot{\Psi}^{kl}(S)\nabla_k\nabla_l\lambda+\phi(t)u+(\alpha-1)u\Psi+(u-\lambda)\dot{\Psi}^{kl}(S)\delta^k_l.\label{lambda}
\end{align} 
 Let  $$b(p,t):=\max\{B^i_j(p,t) \eta_i \eta_j:\sum\limits_{i=1}\limits^{n}\eta^2_i=1\}.$$    We consider 
 \[P(p,t)=\log{b(p,t)}+ A\lambda-Bu, t\in[t_0,\min{T,t_0+\tau}),\]
 where $A$, $B$ are positive constants which will be chosen later.  Suppose that $P$ attains its maximum at $(\overline{p},\overline{t})\in M\times [t_0,\min\{T,t_0+\tau\})$. We choose a local orthonormal frame $e_1, \dots, e_n$ around $\overline{p}$ such that at $(\overline{p},\overline{t})$ we have $g_{i j}=\delta_{i j}$, $\ h_{i j}=\kappa_i \delta_{i j}$ and 
 $$b(\overline{p},\overline{t})=B^1_1(\overline{p},\overline{t}).$$
Let
\[\tilde{P}(p,t)=\log B^1_1(p,t)+A\lambda-Bu.\]
We have $\tilde{P}$ attains maximum at $(\overline{p},\overline{t})$. Moreover, the derivatives of $\tilde{P}$ and $P$ at $(\overline{p},\overline{t})$ coincide. 
%By the classical method of taking locally smooth lower supports, together with choosing an orthonormal frame at a point, we consider the evolution equation of $B^1_1$ at this point at time $t_0$ where in under this orthonormal frame, the value of $B^1_1$ equals the maximum radius of curvature, we have 
By (\ref{Bij}) we have 
\begin{align}
&\frac{\partial}{\partial t} \log{B^1_1}\\=&\frac{1}{B^1_1}\dot{\Psi}^{k l}(S)\nabla_k\nabla_l B^1_1+2\dot{\Psi}^{k l}(S)\nabla_k S^p_1 \nabla_l B^1_p\nonumber\\
&-\ddot{\Psi}^{kl,pq}B^1_1 \nabla_1 S^k_l \nabla^1 S^p_q-\frac{1}{B^1_1}\dot{\Psi}^{k l}(S)(S^2)_{kl} (B^2)^1_1\nonumber\\
&-\frac{1}{B^1_1}(\dot{\Psi}^{k l}(S)(S^2)_{kl}+2\Psi(S)-2\phi)B^1_1\nonumber\\
&-\frac{1}{B^1_1}((1-\alpha)\Psi(S)-\phi-\dot{\Psi}^{k l}(S)\delta^k_l)\nonumber
    \\
=&\dot{\Psi}^{k l}(S)\nabla_k\nabla_l \log{B^1_1}+\frac{\dot{\Psi}^{k l} \nabla_k B^1_1 \nabla_l B^1_1}{(B^1_1)^2}+2\dot{\Psi}^{k l}(S)B^1_1\nabla_k S^p_1 \nabla_l B^1_p\nonumber\\
&-\ddot{\Psi}^{kl,pq}B^1_1 \nabla_1 S^k_l \nabla^1 S^p_q\nonumber\\
&-\dot{\Psi}^{k l}(S)(S^2)_{kl} B^1_1-(\dot{\Psi}^{k l}(S)(S^2)_{kl}+2\Psi(S)-2\phi)\nonumber\\
&-\frac{1}{B^1_1}((1-\alpha)\Psi(S)-\phi-\dot{\Psi}^{k l}(S)\delta^k_l).
\end{align}
Combining the evolution equation (\ref{u}) and (\ref{lambda}), we obtain that at $(\overline{p},\overline{t})$ there holds 
\begin{align*}
&\frac{\partial}{\partial t}P-\dot{\Psi}^{kl}(S)\nabla_k\nabla_l P\\
=&\dot{\Psi}^{k l}(S)\nabla_k\nabla_l \log{B^1_1}+\frac{\dot{\Psi}^{k l} \nabla_k B^1_1 \nabla_l B^1_1}{(B^1_1)^2}+2\dot{\Psi}^{k l}(S)\nabla_k S^p_1 \nabla_l B^1_p\\&-\ddot{\Psi}^{kl,pq}B^1_1 \nabla_1 S^k_l \nabla^1 S^p_q+\phi\left(2+\frac{1}{B^1_1}+Au-B\lambda\right)\\&+\dot{\Psi}^{k l}(S)(S^2)_{kl}(-B^1_1-1-Bu)\\&+\Psi(S)\left(-2-\frac{1}{B^1_1}(1-\alpha)+A(\alpha-1)u-B(2\alpha u-(1+\alpha)\lambda\right)
        \\&+\dot{\Psi}^{k l}(S)\delta^k_l\left(\frac{1}{B^1_1}+(A-B)(u-\lambda)\right).
    \end{align*}
    By a direct calculation, we have \begin{align}\dot{\Psi}^{kl}(S)(S^2)_{kl}=&\alpha\Psi(S)E_1(S),\label{pingfang}\\
\dot{\Psi}^{kl}(S)\delta^k_l=&\frac{\alpha}{n}\Psi(S)\frac{\sigma_{n-1}(S)}{\sigma_n(S)}.\label{lingci} 
\end{align}
    Hence at $(\overline{p},\overline{t})$ there holds
\begin{align*}
&\frac{\partial}{\partial t}P-\dot{\Psi}^{kl}(S)\nabla_k \nabla_l P\nonumber
\\=&\frac{\dot{\Psi}^{k l} \nabla_k B^1_1 \nabla_l B^1_1}{(B^1_1)^2}+2\dot{\Psi}^{k l}(S)\nabla_k S^p_1 \nabla_l B^1_p-\ddot{\Psi}^{kl,pq}B^1_1\nabla_1 S^k_l \nabla^1 S^p_q\nonumber\\
&+\phi\left(2+\frac{1}{B^1_1}+Au-B\lambda\right)\\&+\Psi(-2-\frac{1}{B^1_1}(1-\alpha)+A(\alpha-1)u-B(2\alpha u-(1+\alpha)\lambda)\nonumber\\
&+\alpha E_1(S)(-B^1_1-1-Bu)+\frac{\alpha}{n}\frac{\sigma_{n-1}(S)}{\sigma_n(S)}(\frac{1}{B^1_1}+(A-B)(u-\lambda))).\end{align*}
By Lemma \ref{C0estimate}, there exist constants $c_4>1$, $c_5>0$ depending on $n$ and $M_0$ such that $c_4<\frac{\lambda}{u}$ and $\lambda-u>c_5$. 
Hence we can choose $B>0$ depending on $n$, $M_0$ and $\alpha$ such that $B+1<c_4B-\frac{3}{2c_3}<\frac{\lambda}{u}B-\frac{3}{2c_3}$. Choose a constant $A$ such that $A\in(B+1, c_4B-\frac{3}{2c_3})$. Then by \eqref{ulow} we have
\begin{align*}
    2+Au-B\lambda&<-1,\\
    A-B>1.
\end{align*}
Since our aim is to derive an upper bound on $B^1_1$, without loss of generality we assume that $B^1_1>\frac{1}{c_5}$.
Noticing that $\frac{\sigma_{n-1}(S)}{\sigma_n(S)}\geq B^1_1$, we have 
\begin{align}
&\frac{\partial}{\partial t}P-\dot{\Psi}^{kl}(S)\nabla_k \nabla_l P\nonumber
\\\leq&\frac{\dot{\Psi}^{k l} \nabla_k B^1_1 \nabla_l B^1_1}{(B^1_1)^2}+2\dot{\Psi}^{k l}(S)\nabla_k S^p_1 \nabla_l B^1_p-\ddot{\Psi}^{kl,pq}B^1_1\nabla_1 S^k_l \nabla^1 S^p_q\nonumber\\
&+\phi\left(\frac{1}{B^1_1}-1\right)\\
&+\Psi(-2-\frac{1}{B^1_1}(1-\alpha)+A(\alpha-1)u-B(2\alpha u-(1+\alpha)\lambda)\nonumber\\
&+\alpha E_1(S)(-B^1_1-1-Bu)+\frac{\alpha}{n}B^1_1(\frac{1}{B^1_1}-c_5)).\label{P}\end{align}
 We divide the the right side of equation (\ref{P}) into two parts as follows: \begin{align*}
    Q_1:=&\frac{\dot{\Psi}^{k l} \nabla_k B^1_1 \nabla_l B^1_1}{(B^1_1)^2}+2\dot{\Psi}^{k l}(S)\nabla_k S^p_1 \nabla_l B^1_p-\ddot{\Psi}^{kl,pq}B^1_1\nabla_1 S^k_l \nabla^1 S^p_q\nonumber,\\
    Q_2:=&\phi\left(\frac{1}{B^1_1}-1\right)+\Psi(-2-\frac{1}{B^1_1}(1-\alpha)+A(\alpha-1)u-B(2\alpha u-(1+\alpha)\lambda)\nonumber\\
&+\alpha E_1(S)(-B^1_1-1-Bu)+\frac{\alpha}{n}B^1_1(\frac{1}{B^1_1}-c_5)).
\end{align*}
 By Lemma \ref{C0estimate}, there exists a constant $c_6>\frac{1}{c_5}$ depending on $M_0$, $n$ and $\alpha$ such that $B^1_1>c_6$ implies $Q_2<0$. 
 
 Noticing that $$\begin{aligned}&\dot{E}_n^{kl}(S)=E_n(S)B^k_l, \\ &\ddot{E}_n^{kl,pq}(S)=E_n(S)(B^k_l B^p_q-B^{kp} B_{ql}),\\
& \nabla_k B^1_1=-(B^1_1)^2\nabla_k S_{11},\\
&\nabla_k B^r_1=-B^1_1 B^r_r \nabla_k S_{r1},\\
\end{aligned}$$
 and combining the Codazzi equation, we have 
\begin{align*}
    Q_1&=\frac{\alpha}{n}\tilde{E}_n^\frac{\alpha}{n}\left((B^1_1)^2B^{kk}|\nabla_kS^1_1|^2-B^1_1B^k_kB^p_p|\nabla_kS_{1p}|^2-\frac{\alpha}{n}B^1_1(B^{kk}\nabla^1S_{kk})^2\right)\\
       &\leq \left((B^1_1)^2B^{kk}|\nabla_kS^1_1|^2-B^1_1B^k_kB^p_p|\nabla_kS_{1p}|^2\right) \leq 0.
\end{align*}
By arguments above, if $B^1_1>c_6$, there holds $$\frac{\partial}{\partial t}P-\dot{\Psi}(S)^{kl}\nabla_k\nabla_l P<0.$$
Then we have $P\leq \max\{\max\limits_{p\in M}P(p,t_0), c^6\}$, from which we know that there exists $c_7>0$ depending on $n$, $M_0$ and $\alpha$ such that $$\log b(p,t)\leq \max\{\max\limits_{p\in M}\log b(p,t_0), c^7\}.$$ 
Therefore,  there exists a constant $c_8>0$ depending on $n$, $M_0$ and $\alpha$ such that \begin{align*}
    \max\limits_{(p,t)\in M\times[t_0,\max\{t_0+\tau,T\})} b(p,t)\leq \max\{\max\limits_{p\in M)} b(p,t_0),c_8\} e^{c_8}.
\end{align*}
 Starting from $t=0$ and iterating the above estimates, we obtain the bound required after at most $\left[\frac{T}{\tau}\right]+1$ steps.
\end{proof}

Having established the lower bound on shifted curvatures of $M_t$ on finite time intervals, we now adapt Tso's technique \cite{Tso} to establish the upper bound on $\tilde{E}_n$.
\begin{thm}\label{thmupper}
    Let $M_t=\partial\Omega_t$ be a smooth solution to flow (\ref{flow})  on  $[0,T)$. If $M_0$ is strictly h-convex, then there exists $C_2(T)$ depended on $n$, $M_0$, $\alpha$ and $T$ such that  the shifted principal curvatures of $M_t$ satisfy $$\tilde{E}_n\leq\frac{1}{C_2(T)}, \,\text{for all}\ p\in M_t \, ,i=1, \dots, n.$$ 
\end{thm}
\begin{proof}

For any $t_0\in[0,T)$, let $p_0$ be a center of an inner ball of $M_{t_0}$. By argument above we know that the support function $u$ of $M_0$ with respect to $p_0$  satisfies $2c_3<u$ for any $t\in \left[t_0, \min \left\{T, t_0+\tau\right\}\right)$. 

Let $Q:=\frac{\Psi(S)}{u-c_3}$. We have
\begin{align}
    \nonumber
\frac{\partial}{\partial t} Q=&\frac{1}{u-c_3}\frac{\partial}{\partial t} \Psi(S)-\frac{\Psi}{(u-c_3)^2}\frac{\partial}{\partial t}u\\ \nonumber
    =&\frac{1}{u-c_3}\dot{\Psi}^{kl}\nabla_k\nabla_l\Psi-\frac{\Psi}{(u-c_3)^2}\dot{\Psi}^{kl}\nabla_k\nabla_l u\nonumber\\
    &-\frac{\phi(t)}{u-c_3}\left(\dot{\Psi}^{kl}(S^2)_{kl}+2\alpha\Psi+\cosh d_{p_0} Q\right)\nonumber\\
    &+Q^2\left(-2\alpha c_3+(1+\alpha)\cosh d_{p_0}\right)\nonumber\\
    &-c_3Q\dot{\Psi}^{kl}(S)(S^2)_{kl}+Q\frac{\cosh d_{p_0}-u}{u-c_3}\dot{\Psi}^{kl}\delta^k_l.\label{evoofQ}
\end{align}
 By a direct calculation 
 \begin{align*}
   \dot{\Psi}^{kl}(S)\nabla_k\nabla_l Q=&\frac{1}{u-c_3}\dot{\Psi}^{kl}(S)\nabla_k\nabla_l\Psi(S)\\
   &-\frac{\Psi(S)}{(u-c_3)^2}\dot{\Psi}^{kl}(S)\nabla_k\nabla_l u-\frac{2}{u-c_3}\dot{\Psi}^{kl}(S)\nabla_kQ\nabla_lu,  
 \end{align*}
 we have
 \begin{align*}
     \frac{\partial}{\partial t}Q=&\dot{\Psi}^{kl}(S)\nabla_k\nabla_l Q+\frac{2}{u-c_3}\dot{\Psi}^{kl}(S)\nabla_k Q\nabla_l u \\&-\frac{\phi(t)}{u-c_3}\left(\dot{\Psi}^{kl}(S^2)_{kl}+2\alpha\Psi+\cosh d_{p_0} Q\right)\\&+Q^2\left(-2\alpha c_3+(1+\alpha)\cosh d_{p_0}\right)-c_3Q\dot{\Psi}^{kl}(S)(S^2)_{kl}\\&+Q\frac{\cosh d_{p_0}-u}{u-c_3}\dot{\Psi}^{kl}\delta^k_l\\
     \leq& \dot{\Psi}^{kl}(S)\nabla_k\nabla_l Q+\frac{2}{u-c_3}\dot{\Psi}^{kl}(S)\nabla_k Q\nabla_l u+Q^2(1+\alpha) \cosh d_{p_0}\\&+Q\frac{\cosh d_{p_0}}{(u-c_3)^2}\dot{\Psi}^{kl}(S)\delta^k_l-Q\frac{c_3}{u-c_3}\dot{\Psi}^{kl}(S)(S^2)_{kl}.
 \end{align*}
Plugging \eqref{pingfang} and \eqref{lingci} into the equation and using Lemma \ref{C0estimate} together with \eqref{u}, we conclude that there exist constants $c_9>0$ and $c_{10}>0$ depending on $n$, $M_0$ and $\alpha$ such that \begin{align}
   \nonumber \frac{\partial}{\partial t}Q \nonumber \leq &\dot{\Psi}^{kl}(S)\nabla_k\nabla_l Q+\frac{2}{u-c_3}\dot{\Psi}^{kl}(S)\nabla_k Q\nabla_l u\\&+Q^2((1+\alpha)\cosh d_{p_0}+c_9\frac{\sigma_{n-1}}{\sigma_n}-c_{10} Q^\frac{1}{\alpha})\nonumber\\
\leq&\dot{\Psi}^{kl}(S)\nabla_k\nabla_l Q+\frac{2}{u-c_3}\dot{\Psi}^{kl}(S)\nabla_k Q\nabla_l u\nonumber\\
&+Q^2((1+\alpha)\cosh d_{p_0}+nc_9C_1(T)-c_{10} Q^\frac{1}{\alpha}).\label{upperboundofPhi}
\end{align} 
Then there exists a constant $c(T)$ depending on $M_0$, $n$, $\alpha$ and $T$ such that $Q>c(T)$ implies $$\frac{\partial}{\partial t}Q-\dot{\Psi}^{kl}\nabla_k\nabla_l Q-\frac{2}{u-c_3}\dot{\Psi}^{kl}(S)\nabla_k Q\nabla_l u\leq -Q^2.$$ Then, by using maximum principle, a standard  argument  yields the required estimates.
\end{proof}

Let $M_t$ be a solution to the flow \eqref{flow} with strictly h-convex initial hypersurface $M_0$.  For any finite time interval $[0,T)$, combining Lemma \ref{C0estimate}, Theorem \ref{thmlower} and Theorem \ref{thmupper}, we can obtain the $C^0$, $C^1$ and $C^2$ estimates of $M_t$ depending on $n, M_0,\alpha$ and $T$.  By the inverse concavity of curvature function $\tilde E_n^\frac{1}{n}$, we can apply similar arguments in \cites{AndrewsWei2018} to obtain the higher order regularity depending on  $n, M_0,\alpha$ and $T$. Then, by a classical extension argument (see \cite[section 4]{WeiYangZhou2025}), we obtain the following long-time existence of the flow \eqref{flow}.
\begin{thm}
     Let  $M_t$ be a smooth strictly h-convex solution to flow \eqref{flow}. Then the solution $M_t$ exists for all time $t\in [0,\infty)$.
\end{thm}

\subsection{Hausdorff Convergence and smooth convergence}\label{convergenceofflow}
In this subsection we adapt the technique in \cite{WeiYangZhou2024} to obtain the convergence of the strictly h-convex solution $M_t$ to the flow (\ref{flow}). We note that once the subsequential Hausdorff convergence is established, we can obtain the smooth and exponential convergence of the solution to the flow (\ref{flow}) by similar arguments in \cite[sections 6 and 7]{WeiYangZhou2024}. We first derive the monotonicity of the modified quermassintegral $\tilde{W}_n(\Omega_t)$ and the uniform lower and upper bounds on the area average of shifted Gauss curvature denoted by \begin{align*}
    \overline{\tilde{E}_n}=\frac{\int_{M_t} \tilde{E}_n d\mu_{g_t}}{|M_t|}.
\end{align*}.  %Then following the convergence argument in \cite{WeiYangZhou2024} we have the smooth and exponential convergence of $M_t$. 
\subsubsection{Uniform lower and upper bounds on $\overline{\tilde{E}}_n$}
We first derive the monotonicity of $\tilde{W}_n$.
\begin{lemma}\label{monoofe}
     Let $M_t=\partial \Omega_t$ be a smooth strictly  h-convex solution to (\ref{flow}).  Then $\tilde{W}_n\left(\Omega_t\right)$ is monotone decreasing in time $t$, which is strictly decreasing unless $\Omega_t$ is a geodesic ball.
\end{lemma}
\begin{proof}
    By the evolution equation (\ref{evolutionquer}) and the expression (\ref{phi}) of $\phi$, we have 
    \begin{align*}
     \frac{\partial}{\partial t}\tilde{W}_n(\Omega_t)&=\int_{M_t} \tilde{E}_n(\phi-\tilde{E}^\frac{\alpha}{n}_n)d \mu_{g_t}\\
     &=\frac{1}{| M_t|}\int_{M_t} \tilde{E}_n\left(\int_{M_t}\tilde{E}_n^{\frac{\alpha}{n}}d\mu_{g_t}-\tilde{E}^\frac{\alpha}{n}_n\right)d \mu_{g_t}\\
     &=\int_{M_t} \tilde{E}_n^\frac{\alpha}{n}\left( \frac{\int_{M_t}\tilde{E}_n d\mu_{g_t}}{|M_t|}-\tilde{E}_n\right)d\mu_{g_t}\\
     &=-\int_{M_t} \left(\tilde{E}_n^\frac{\alpha}{n}-\left(\overline{\tilde{E}_n}\right)^\frac{\alpha}{n}\right)\left( \tilde{E}_n-\overline{\tilde{E}_n}\right)d\mu_{g_t}\\
     &\leq 0,   
    \end{align*}
where equality holds if and only if $\tilde{E}_k$ is constant on $M_t$, i.e., $M_t$ is a geodesic sphere.
\end{proof}

The following expression for the modified quermassintegral of geodesic balls of radius $r$ can be found in \cite{Li-Xu}. For the convenience of the reader, we present the  proof below.
\begin{lemma}
    Let $B_r$ be a geodesic ball of radius $r$ centered at $p\in\mathbb{H}^{n+1}$. We have \begin{align}
        \tilde{W}_n(B_r)=\frac{\omega_n}{n}(1-e^{-nr}).\label{modifiedquermassforB}
    \end{align}
\end{lemma}
\begin{proof}
    By \cite[Lemma 2.4]{AndrewsChenWei2021}, we have 
    \begin{align*}
     \frac{\partial}{\partial r}\tilde{W}_n(B_r)=\int_{\partial B_r}\tilde{E}_n d\mu_{g_r}=\omega_ne^{-nr}.   
    \end{align*}
    By (\ref{Quermassine}) we have $$\lim\limits_{r\rightarrow 0^+}\tilde{W}_n(B_r)=0.$$
    Hence we have 
    \begin{align*}
        \tilde{W}_n(B_r)=\int_0^r\omega_ne^{-ns}ds=\frac{\omega_n}{n}(1-e^{-nr}).
    \end{align*}
\end{proof}

Now we derive uniform upper and lower bounds on $\overline{\tilde{E}_n}$.
\begin{lemma}\label{averofG}
    Let $M_t=\partial \Omega_t$ be a  strictly  h-convex solution to the flow  (\ref{flow}). Then there exists a constant $c_{11}>0$ depending on $n$, $M_0$ such that \begin{align}
        \frac{1}{c_{11}}<\overline{\tilde{E}_n}<c_{11}.\label{uniformestofE}
    \end{align}
\end{lemma}
\begin{proof}
    By Theorem \ref{thmlower}, we have $\Omega_t$ is strictly h-convex for any $t\geq 0$. Hence by Lemma \ref{C0estimate} we know that for any $t\geq 0$ there exists a geodesic ball of radius $c_2$ which contains $\Omega_t$. By Theorem \ref{mono} and Lemma \ref{modifiedquermassforB}, we have $$\tilde{W}_n(\Omega_t)\leq \tilde{W}_n(B_{c_2})\leq \frac{\omega_n}{n}(1-e^{-nc_2}).$$ Then, by (\ref{modifiedcurvaturemeasure}),  for any $t\geq 0$ there holds $$\frac{\omega_n e^{-n c_2}}{|M_t|}\leq\overline{\tilde{E}_n}(t)\leq \frac{\omega_n}{|M_t|}.$$ By the strict h-convexity of $M_t$, Theorem \ref{thmmonoofquer} and Lemma \ref{C0estimate}, we have $|\partial B_{c_1}|\leq|M_t|\leq|\partial B_{c_2}|$,  which yields $$\frac{\omega_n e^{-n c_2}}{|\partial B_{c_2}|}\leq\overline{\tilde{E}_n}(t)\leq \frac{\omega_n}{|\partial B_{c_1}|}.$$
\end{proof}
\begin{rmk}
    By the modified Minkowski formula \eqref{shiftedminkow} and Lemma \ref{C0estimate}, there exists a constant $c_{12}>0$ depending on $n$, $M_0$ such that  $$\frac{1}{c_{12}}<\frac{\int_{M_t} \tilde{E}_kd\mu_{g_t}}{|M_t|}<c_{12}, \,\forall 1\leq k\leq n.$$
\end{rmk}
\subsubsection{Subsequential Hausdorff convergence}
Having derived the monotonicity of $\tilde{W}_n$ in Lemma \ref{monoofe} and established the uniform lower and upper bounds on $\overline{\tilde{E}}_n$ in Lemma  \ref{averofG}, we can adapt similar arguments in \cite[Lemma 6.2]{WeiYangZhou2024}  to  obtain the following $\mathcal{L}^1$ oscillation decay of shifted Gauss curvature. 
\begin{lemma}\label{subconvergeofE}
    Let $M_t=\partial \Omega_t$ be a smooth strictly  h-convex solution to (\ref{flow}). Then there exists a subsequence of $\{t_i\}\subseteq [0,\infty)$ with $t_i\rightarrow +\infty$ such that
    \begin{align}
        \int_{M_{t_i}}\left|\tilde{E}_n-\overline{\tilde{E}_n}\right|d \mu_{g_{t_i}}\rightarrow 0, \quad \text{as}\ i\rightarrow\infty.
    \end{align}
\end{lemma}
Now, adapting an argument similar to that in \cite{WeiYangZhou2024}, we can derive the subsequential Hausdorff convergence of $M_t$ and characterize the limit shape by combining Theorem \ref{hausdorffofhconvex} and Theorem \ref{modifiedcurvature}, which are adapted to the h-convex setting.
\begin{thm}
    Let $M_t=\partial \Omega_t$ be a smooth strictly  h-convex solution to the flow \eqref{flow}. Then there exists a subsequence of $\{t_i\}\subseteq [0,\infty)$ with $t_i\rightarrow +\infty$ such that $M_{t_i}$ converge to a geodesic sphere $\partial B_{\rho_{\infty}}(\tilde{p})$ in Hausdorff sense, where $\tilde{p}$ is the center of $B_{\rho_{\infty}}(\tilde{p})$ and $\mathrm{Vol}(B_{\rho_{\infty}}(\tilde{p}))=\mathrm{Vol}(\Omega_0)$.
\end{thm}
\begin{proof}
    Let $\{t_i\}$ be a sequence satisfying the condition stated in Lemma \ref{subconvergeofE}. By an   Alexandrov reflection argument in \cite[Lemma 4.1]{AndrewsChenWei2021}, we know that $\{\Omega_{t_i}\}$ is a bounded subset of $\mathcal{H}_{n+1}(\mathbb{H}^{n+1})$. By Lemma \ref{C0estimate} and Theorem \ref{hausdorffofhconvex},  there exist $\Omega_{t_\infty}\in\mathcal{H}_{n+1}(\mathbb{H}^{n+1})$ and a subsequence $\{t_{i_k}\}\subseteq\{t_i\}$ such that $\Omega_{t_{i_k}}\rightarrow \Omega_{t_\infty}$ in Hausdorff sense as $i_k\rightarrow \infty$.  Moreover, by Lemma \ref{averofG}, we can assume that there exists $c>0$ such that the $\overline{\tilde{E}_n}(t_{i_k})$ converges to $c$.
     By (\ref{tildeV}) and Theorem \ref{weakcon}, we can obtain the weak continuity of the modified curvature measures.
    Hence, for any bounded continuous function $f$ on $\mathbb{H}^{n+1}$ with compact support, we have 
     \begin{align*}
           &\left|\int f d \tilde{V}_0(\Omega_{t_{i_k}},\cdot)-c \int f d \tilde{V}_n(\Omega_{t_{i_k}},\cdot)\right| \\= & \left|\int_{M_{t_{i_k}}} f \tilde{E}_n d \mu_{g_{t_{i_k}}}-\int_{M_{t_{i_k}}} c f d \mu_{g_{t_{i_k}}}\right| \\ \leq & \sup |f| \int_{M_{t_{i_k}}}|\tilde{E}_n-c| d \mu_{g_{t_{i_k}}} \\ \leq & \sup |f| \int_{M_{t_{i_k}}}|\tilde{E}_n-\overline{\tilde{E}_n}| d \mu_{g_{t_{i_k}}}\\&+\sup |f| \int_{M_{t_{i_k}}}|\overline{\tilde{E}_n}-c| d \mu_{g_{t_{i_k}}} \rightarrow 0,\, \text{as}\  i_k\rightarrow \infty,
     \end{align*}
where the limit in the final step is guaranteed by Lemma \ref{subconvergeofE} and the uniform upper bound on $|M_{t_{i_k}}|$. 
Hence we have $$\int f d \tilde{V}_0(\Omega_{t_\infty},\cdot)=c\int f d \tilde{V}_n(\Omega_{t_\infty},\cdot),$$ i.e., $c\tilde{V}_n(\Omega_{t_\infty},\cdot)=\tilde{V}_0(\Omega_{t_\infty},\cdot)$. By Theorem \ref{modifiedcurvature}, we know that $\Omega_{t_\infty}$ is a geodesic ball. Since the volume of $\Omega_t$ is preserved along the flow \eqref{flow}, this completes the proof.
\end{proof}

Let $M_t$ be a smooth strictly h-convex solution to the flow \eqref{flow}. 
Having established the subsequential Hausdorff convergence of $M_t$, we can apply the argument in \cite[Section 6]{WeiYangZhou2024} to  obtain the convergence of the centers of the inner balls of $M_t$, from which we can obtain the uniform upper and lower bounds on the shifted principal curvatures of $M_t$ (see \cite{WeiYangZhou2024} for details). Then, following \cite[Section 7]{WeiYangZhou2024}, we complete the proof of Theorem \ref{mainthm}.

\section*{Acknowledgements}
	The author would like to thank Prof. Yong Wei for discussions. This work is supported by National Key Research and Development Program of China 2021YFA1001800.

%----------------------------------------------------------

%\bibliographystyle{plainnat}
%\bibliography{reference.bib}

\end{document}